\documentclass[10pt, reqno]{amsart}
\usepackage{graphicx, amssymb, amsmath, amsthm}
\numberwithin{equation}{section}

\usepackage{color}
\usepackage{epsfig}
\usepackage{subfigure}
\usepackage{tikz}
\usepackage[backref=page]{hyperref}

\hypersetup{urlcolor=blue, citecolor=red}

\usepackage{hyperref}

\let\Re=\undefined\DeclareMathOperator*{\Re}{Re}
\let\Im=\undefined\DeclareMathOperator*{\Im}{Im}

\newcommand{\R}{\mathbb{R}}
\newcommand{\C}{\mathbb{C}}

\newcommand{\Z}{\mathbb{Z}}

\newcommand{\HH}{\mathcal{H}}

\newtheorem{theorem}{Theorem}[section]

\newtheorem{lemma}[theorem]{Lemma}

\newtheorem{proposition}[theorem]{Proposition}

\theoremstyle{definition}
\newtheorem{definition}[theorem]{Definition}
\newtheorem{remark}[theorem]{Remark}

\makeatletter
\newcommand{\Extend}[5]{\ext@arrow0099{\arrowfill@#1#2#3}{#4}{#5}}
\makeatother

\begin{document}
\title[wave equation in magnetic fields]{Decay estimates for one Aharonov-Bohm solenoid in a uniform magnetic field II: wave equation}

\author{Haoran Wang}
\address{Department of Mathematics, Beijing Institute of Technology, Beijing 100081;}
\email{wanghaoran@bit.edu.cn}

\author{Fang Zhang}
\address{Department of Mathematics, Beijing Institute of Technology, Beijing 100081;}
\email{zhangfang@bit.edu.cn}

\author{Junyong Zhang}
\address{Department of Mathematics, Beijing Institute of Technology, Beijing 100081;}
\email{zhang\_junyong@bit.edu.cn}

\begin{abstract}
This is the second of a series of papers in which we investigate the decay estimates for dispersive equations with Aharonov-Bohm solenoids in a uniform magnetic field.
In our first starting paper \cite{WZZ}, we have studied the Strichartz estimates for Schr\"odinger equation with one Aharonov-Bohm solenoid  in a uniform magnetic field. The wave equation in this setting becomes more delicate since
a difficulty is raised from the square root of the eigenvalue of the Schr\"odinger operator $H_{\alpha, B_0}$ so that we cannot directly construct the half-wave propagator.
An independent interesting result concerning the Gaussian upper bounds of the heat kernel is proved by using two different methods. The first one is based on establishing Davies-Gaffney inequality in this setting
and the second one is straightforward to construct the heat kernel (which efficiently captures the magnetic effects) based on the Schulman-Sunada formula.
As byproducts, we prove optimal bounds for the heat kernel and show the Bernstein inequality and the square function inequality for Schr\"odinger operator with one Aharonov-Bohm solenoid in a uniform magnetic field.
\end{abstract}

%\bigskip\bigskip
\maketitle

\begin{center}
 \begin{minipage}{120mm}
 \end{minipage}
 \end{center}

\section{Introduction}

In this paper, as a sequence of recent papers \cite{FZZ22,GYZZ22,WZZ}, we study the decay and Strichartz estimates for the wave equation on the plane pierced by one infinitesimally thin Aharonov-Bohm solenoid and subjected to a perpendicular uniform magnetic field of constant magnitude $B_0$.
More precisely, we study the wave equation
\begin{equation}\label{eq:wave}
\begin{cases}
\partial_{tt} u(t,x)+H_{\alpha,B_0} u(t,x)=0,\\
u(0,x)=u_0(x),\quad \partial_t u(0, x)=u_1(x),
\end{cases}
\end{equation}
where the magnetic Schr\"odinger operator
\begin{equation}\label{H-A}
H_{\alpha,B_0}=-(\nabla+i(A_B(x)+A_{\mathrm{hmf}}(x)))^2,
\end{equation} 
is the same as the one considered in \cite{WZZ}.
Here, $A_B(x)$ is the Aharonov-Bohm potential (initially introduced in \cite{AB59})
\begin{equation}\label{AB-potential}
A_B(x)=\alpha\Big(-\frac{x_2}{|x|^2},\frac{x_1}{|x|^2}\Big),\quad x=(x_1,x_2)\in\mathbb{R}^2\setminus\{0\},
\end{equation}
where $\alpha\in\mathbb{R}$ represents the circulation of $A_B$ around the solenoid; $A_{\mathrm{hmf}}(x)$ is given by
\begin{equation}\label{A-hmf}
A_{\mathrm{hmf}}(x)=\frac{B_0}{2}(-x_2,x_1),\quad B_0>0,
\end{equation}
which generates the background uniform magnetic field.

We stress that the model is on $\R^2$ and the magnetic field $B$ is given by
\begin{equation}\label{B-n}
B(x):=DA-DA^t,\quad B_{ij}=\frac{\partial A^i}{\partial x_j}-\frac{\partial A^j}{\partial x_i},\quad i,j=1,2.
\end{equation}
Hence, the generated magnetic field $B(x)=B_0+\alpha \delta(x)$ is actually a superposition of the uniform field and the Aharonov-Bohm field, where $\delta$ is the usual Dirac delta.
As mentioned in \cite{WZZ}, the Aharonov-Bohm potential that produces the singular magnetic field has the same homogeneity as $\nabla$ (homogenous of degree $-1$) so that the perturbation from the Aharonov-Bohm potential \eqref{AB-potential} is critical; the potential $A_{\mathrm{hmf}}(x)$ is unbounded at infinity and the uniform magnetic filed $B(x)=B_0$ from \eqref{B-n} generates a trapped well. Moreover, due to the presence of the potential \eqref{A-hmf}, the spectrum of the operator $H_{\alpha,B_0}$ consists of pure point, and thus the dispersive behavior of wave equation associated with $H_{\alpha,B_0}$ will be distinguished from the models in \cite{FZZ22,GYZZ22}.  \vspace{0.2cm}

The Hamiltonian $H_{\alpha,B_0}$ can be defined as a self-adjoint operator on $L^2$, via Friedrichs' Extension Theorem (see e.g. \cite[Thm. VI.2.1]{K} and \cite[X.3]{RS}), with a natural form domain, which in 2D turns out to be equivalent to
$$
\mathcal D(H_{\alpha,B_0})\simeq \HH^1_{\alpha,B_0}:=\left\{f\in L^2(\R^2;\C):\int_{\R^2}\left|\big(\nabla +i(A_B+A_{\mathrm{hmf}}\big)f\right|^2\,dx<+\infty\right\}.
$$
We refer to \cite[Section 2]{WZZ} for the Friedrichs' extension via quadratic forms and to \cite{ESV} for  more about the self-adjoint extension theory. In what follows and throughout, the operator $H_{\alpha,B_0}$ should be regarded as a self-adjoint operator generated by the procedure of the Friedrichs' extension. Therefore, the half-wave propagator $e^{it\sqrt{H_{\alpha,B_0}}}$ can be treated as one-parameter groups of operators on $L^2(\R^2)$.
This allows to study a large class of dispersive estimates, such as time-decay (perhaps local in time), Strichartz and local smoothing for dispersive evolutions as \eqref{eq:wave}. The validity of such properties has been central object of deep investigation of dispersive equations in the last decades, due to their relevance in the description of linear and nonlinear dynamics. To better frame our results, let us briefly sketch the state of the art about these problems.

\vspace{0.2cm}

Due to the significance of dispersive and Strichartz estimates in harmonic analysis and partial differential equations, there are too much literature to cite all of them here. But we would like to refer to \cite{BPSS,BPST,CS,DF,DFVV,EGS1,EGS2,S} and the references therein for various dispersive equations with electromagnetic potentials in mathematics and physics. The dispersive equations with the Aharonov-Bohm potential, as a diffraction physical model, have attracted more and more researchers to study from the mathematical perspective. In \cite{FFFP,FFFP1}, the authors studied the validity of the time decay estimates for the Schr\"odinger equation with the Aharonov-Bohm potential. However, due to the lack of pseudo-conformal invariance (which plays a critical role in the Schr\"odinger case), the arguments of \cite{FFFP,FFFP1} break down for the wave equation. Very recently, Fanelli, Zheng and the last author \cite{FZZ22} established Strichartz estimate for the wave equation by constructing the odd sine propagator. To solve open problems, raised in the survey \cite{Fanelli} on the dispersive estimates for other equations (e.g. Klein-Gordon, Dirac, etc.), Gao, Yin, Zheng and the last author \cite{GYZZ22} constructed the spectral measure and then applied to prove the time decay and Strichartz estimates for the Klein-Gordon equation.
The potential models in \cite{FFFP,FFFP1,FZZ22,GYZZ22} are all scaling-invariant and without unbounded (at infinity) perturbations, which is a special case of our model \eqref{H-A} (with $B_0\equiv 0$). In this paper, as \cite{WZZ}, we proceed to consider the wave equation in the magnetic fields mixed with the Aharonov-Bohm and the uniform ones.\vspace{0.2cm}

 %%%%%%%%%%%%%%%%%%%%
%%%%%MAIN THEOREMS%%%%%%
%%%%%%%%%%%%%%%%%%%%

Before stating our main results, let us introduce some preliminary notations.
We define the magnetic Besov spaces as follows. Let $\varphi\in C_c^\infty(\mathbb{R}\setminus\{0\})$ satisfy $0\leq\varphi\leq 1,\text{supp}\,\varphi\subset[1/2,1]$, and
\begin{equation}\label{LP-dp}
\sum_{j\in\Z}\varphi(2^{-j}\lambda)=1,\quad \varphi_j(\lambda):=\varphi(2^{-j}\lambda), \, j\in\Z,\quad \phi_0(\lambda):=\sum_{j\leq0}\varphi(2^{-j}\lambda).
\end{equation}

\begin{definition}[Magnetic Besov spaces associated with $H_{\alpha,B_0}$]\label{def:besov} For $s\in\R$ and $1\leq p,r<\infty$, the homogeneous Besov norm $\|\cdot\|_{\dot{\mathcal{B}}^s_{p,r}(\R^2)}$ is defined by
\begin{equation}\label{Besov}
\|f\|_{\dot{\mathcal{B}}^s_{p,r}(\R^2)}=\Big(\sum_{j\in\Z}2^{jsr}\|\varphi_j(\sqrt{H_{\alpha,B_0}})f\|_{L^p(\R^2)}^r\Big)^{1/r}.
\end{equation}
In particular, for $p=r=2$, we have the Sobolev norm
\begin{equation}\label{Sobolev1}
\begin{split}
\|f\|_{\dot \HH^s_{\alpha,B_0}(\R^2)}:=\|f\|_{\dot{\mathcal{B}}^s_{2,2}(\R^2)}.
\end{split}
\end{equation}

\end{definition}

\begin{remark}  Alternatively,  the Sobolev space can be defined by
\begin{equation*}
\dot \HH^{s}_{\alpha,B_0}(\R^2):= H_{\alpha,B_0}^{-\frac s2}L^2(\R^2),
\end{equation*}
with the norm
\begin{equation}\label{Sobolev2}
\begin{split}
\|f\|_{\dot \HH^s_{\alpha,B_0}(\R^2)}&:=\big\|H_{\alpha,B_0}^{\frac s2}f\big\|_{L^2(\R^2)}.
\end{split}
\end{equation}
By the spectral theory of operators on $L^2$, the norms in \eqref{Sobolev1} and \eqref{Sobolev2} are equivalent; see Proposition \ref{prop:E-S-B} below.
\end{remark}

\begin{definition}\label{ad-pair}
A pair $(q,p)\in [2,\infty]\times [2,\infty)$ is said to be admissible, if $(q,p)$ satisfies
\begin{equation}\label{adm}
\frac{2}q\leq\frac{1}2-\frac1p.
\end{equation}
For $s\in\R$, we denote $(q,p)\in \Lambda^W_{s}$ if $(q,p)$ is admissible and satisfies
\begin{equation}\label{scaling}
\frac1q+\frac {2}p=1-s.
\end{equation}
\end{definition}
Now we state our main theorem.
\begin{theorem}\label{thm:dispersive} Let $H_{\alpha,B_0}$ be as in \eqref{H-A} and $t\in I:=[0,T]$ with any finite $T$.
 Then there exists a constant $C_T$ depending on $T$ such that
\begin{equation}\label{dis-w}
\|e^{it\sqrt{H_{\alpha,B_0}}}f\|_{L^\infty(\R^2)}\leq C_T  |t|^{-1/2}\|f\|_{\dot{\mathcal{B}}^{3/2}_{1,1}(\R^2)}, \quad t\in I, \quad  t\neq 0.
\end{equation}
Let $u(t,x)$ be the solution of \eqref{eq:wave} with initial data $(u_0,u_1)\in  \dot\HH_{\alpha,B_0}^{s}(\R^2)\times \dot\HH_{\alpha,B_0}^{s-1}(\R^2)$, then the Strichartz estimates 
\begin{equation}\label{stri:w}
\|u(t,x)\|_{L^q(I;L^p(\R^2))}\leq C_T\left(\|u_0\|_{ \dot\HH_{\alpha,B_0}^s(\R^2)}+\|u_1\|_{
\dot\HH_{\alpha,B_0}^{s-1}(\R^2)}\right)
\end{equation}
hold for $(q,p)\in \Lambda^W_{s}$ and $0\leq s<1$.
\end{theorem}

\begin{remark} The local-in-time decay estimate \eqref{dis-w} is quite different from the Schr\"odinger counterpart
(see \cite[Theorem 1.1]{WZZ})
\begin{equation*}
\begin{split}
&\big\|e^{itH_{\alpha,B_0}} f\big\|_{L^\infty(\R^2)}\leq C |\sin(tB_0)|^{-1}\big\| f\big\|_{L^1\R^2)},\quad t\neq \frac{k\pi}{B_0}, \,k\in \Z,
\end{split}
\end{equation*}
which is similar to the harmonic oscillators (see Koch and Tataru \cite{KT05}). The period $\pi/B_0$ is essentially the Larmor period. However, for the wave equation,
provided that the data $f=\varphi(2^{-j}\sqrt{H_{\alpha,B_0}})f$ is localized at frequency scale $2^j$ with $j\in\Z$, we can prove (see \eqref{<1} below)
\begin{equation*}
\begin{split}
\big\|\varphi(2^{-j}\sqrt{H_{\alpha,B_0}})&e^{it\sqrt{H_{\alpha,B_0}}}f\big\|_{L^\infty(\R^2)}\\&\lesssim 2^{2j}\big(1+2^jt\big)^{-N}\|\varphi(2^{-j}\sqrt{H_{\alpha,B_0}}) f\|_{L^1(\R^2)}, \quad\text{for}\quad 2^j t\lesssim 1
\end{split}
\end{equation*}
and (see \eqref{est: mic-decay1} below)
\begin{equation*}
\begin{split}
&\big\|\varphi(2^{-j}\sqrt{H_{\alpha,B_0}})e^{it\sqrt{H_{\alpha,B_0}}}f\big\|_{L^\infty(\R^2)}\\&\lesssim 2^{2j}\big(1+2^jt\big)^{-\frac12}\|\varphi(2^{-j}\sqrt{H_{\alpha,B_0}}) f\|_{L^1(\R^2)},
\quad\text{for}\quad 2^j t\lesssim 1,\quad
 2^{-j}|t|\leq \frac{\pi}{8B_0}.
\end{split}
\end{equation*}
The decay estimates for waves depend on the frequency.
The Strichartz estimates \eqref{stri:w} is still local-in-time but the endpoint of the time interval $T$ can beyond $\frac{\pi}{2B_0}$ which is the upper bound of $T$ for Schr\"odinger's Strichartz estimates.
 Due to the unbounded potentials caused a trapped well, the Strichartz estimate is impossible to be global-in-time (for example, the Strichartz estimates for dispersive equations on sphere or torus), but still captures integration regularity behavior near $t=0$.
\end{remark}
Now let us figure out some points in our proof.\vspace{0.2cm}

\begin{itemize}

\item As mentioned above,  for the Schr\"odinger equation considered in \cite{WZZ}, the explicit eigenvalues and eigenfunctions of the operator $H_{\alpha,B_0}$ are the key ingredients. In particular, the eigenvalues are given by
\begin{equation*}
\lambda_{k,m}=(2m+1+|k+\alpha|)B_0+(k+\alpha)B_0,\quad m,\,k\in\mathbb{Z},\, m\geq0,
\end{equation*}
see \eqref{eigen-v} below. One feature of $\lambda_{k,m}$ is that $k$ and $m$ can be separated in the series convergent argument.  However, for the half-wave propagator, this feature breaks down for the square root of $\lambda_{k,m}$. Therefore, we cannot directly construct the wave propagator by following the argument of \cite{FZZ22,WZZ}. \vspace{0.1cm}

\item Due to the uniform magnetic field caused a trapped well, the spectral measure will involve a factor $\sin(tB_0)$ (which is a short-time decay but not long-time). This lead to the failure of the spectral measure argument in \cite{GYZZ22}. \vspace{0.1cm}

\item To go around constructing the spectral measure, we turn to prove the Bernstein inequality to deal with the low frequency.
For the high frequency, we use the classical subordination formula
\begin{equation*}
e^{-y\sqrt{H_{\alpha,B_0}}}=\frac{y}{2\sqrt{\pi}}\int_0^\infty e^{-sH_{\alpha,B_0}}e^{-\frac{y^2}{4s}}s^{-\frac{3}{2}}ds,\quad y>0,
\end{equation*}
which provides  a connecting bridge between the Schr\"odinger propagator and the half-wave propagator. This idea is originated from \cite{MS} and \cite{DPR}. The dispersive estimates proved in \cite{WZZ} are then used to address the high frequency of the waves. \vspace{0.1cm}

\item The Littlewood-Paley theory (including Bernstein inequality and the square function inequality) associated with the Schr\"odinger operator $H_{\alpha,B_0}$ are proved by establishing
the Gaussian upper bounds for the heat kernel. \vspace{0.1cm}

\item The heat kernel estimates for magnetic Schr\"odinger operators have its own interest, we provide two methods to study the heat kernel. Unfortunately, due to the fact $A(x)=A_{\mathrm{hmf}}(x)+A_B(x)\notin L^2_{\mathrm{loc}}(\R^2)$, Simon's diamagnetic pointwise inequality (see e.g.  \cite[Theorem B.13.2]{Simon}, \cite{AHS1})  cannot be directly used. Even though we cannot recover all the magnetic effects to prove the optimal heat kernel estimates, but we can prove
\begin{equation*}
\big|e^{-tH_{\alpha,B_0}}(x,y)\big|\lesssim  \frac{1}{t}e^{-\frac{|x-y|^2}{Ct}},
\end{equation*}
which is enough for proving the Bernstein inequality and the square function inequality.
We first prove the on-diagonal estimates and then extend to the off-diagonal estimates by establishing the Davies-Gaffney inequality.
The key points are the argument of \cite{CS08} and \cite{Grig} applying to the magnetic operator $H_{\alpha,B_0}$.
To recover more magnetic effects, we use the Schulman-Sunada formula from \cite{stov,stov1} to construct the heat kernel and prove
\begin{equation*}
\Big| e^{-t H_{\alpha, B_0}}(x,y)\Big|\leq C\frac{B_0 e^{-\alpha tB_0}}{4\pi \sinh (tB_0)} e^{-\frac{B_0|x-y|^2}{4\tanh (tB_0)}},
\end{equation*}
which is better than the previous one.
For more discussion on the heat kernel estimates, we refer to the remarks in Section \ref{sec:heat}.

\end{itemize}

 \vspace{0.2cm}

The paper is organized as follows. In Section \ref{sec:pre}, as a preliminary step, we briefly recall the self-adjoint extension and the spectrum of the operator $H_{\alpha,B_0}$, and prove  the equivalence between Sobolev norm and a special Besov  norm. In Section \ref{sec:heat}, we construct the heat kernel and prove the Gaussian upper bounds. In Section \ref{sec:BS}, we prove the Bernstein inequalities and the square function inequality
by using the heat kernel estimates.  Finally, in Section \ref{sec:decay} and Section \ref{sec:str}, we prove the dispersive estimate \eqref{dis-w} and the Strichartz estimate \eqref{stri:w} in Theorem \ref{thm:dispersive} respectively.
\vspace{0.2cm}

{\bf Acknowledgments:}\quad  The authors thank L. Fanelli, P. \v{S}t'ov\'{\i}\v{c}ek and P. D'Ancona for helpful discussions. This work is supported by National Natural Science Foundation of China (12171031, 11901041, 11831004).
\vspace{0.2cm}

\section{preliminaries} \label{sec:pre}

In this section, we first repeat the preliminary section of \cite{WZZ} to recall two known results about the Friedrichs self-adjoint extension of the operator $H_{\alpha,B_0}$ and
the spectrum of  $H_{\alpha,B_0}$. Next, we use the spectral argument to prove the equivalence between the Sobolev norm and a special Besov  norm.

\subsection{Quadratic form and self-adjoint extension}

Define the space $\mathcal{H}_{\alpha,B_0}^1(\R^2)$ as the completion of $\mathcal{C}_c^\infty(\R^2\setminus\{0\};\C)$
with respect to the norm
\begin{equation*}
\|f\|_{\mathcal{H}_{\alpha,B_0}^1(\R^2)}=\Big(\int_{\R^2}|\nabla_{\alpha,B_0} f(x)|^2 dx\Big)^{\frac12}
\end{equation*}
where
\begin{equation*}
\nabla_{\alpha,B_0} f(x)=\nabla f+i(A_B+A_{\mathrm{hmf}})f.
\end{equation*}
The quadratic form $Q_{\alpha,B_0}$ associated with $H_{\alpha,B_0}$ is defined by
\begin{equation*}
\begin{split}
Q_{\alpha,B_0}: & \quad \quad \mathcal{H}_{\alpha,B_0}^1\to \R\\
Q_{\alpha,B_0}(f)&=\int_{\R^2}|\nabla_{\alpha,B_0} f(x)|^2dx.
\end{split}
\end{equation*}
Then the quadratic form $Q_{\alpha,B_0}$ is positive definite, which implies that the operator
$H_{\alpha,B_0}$ is symmetric semi bounded from below and thus admits a self-adjoint extension (Friedrichs extension)
$H^{F}_{\alpha,B_0}$
with the natural form domain
\begin{equation*}
\mathcal{D}=\Big\{f\in \mathcal{H}_{\alpha,B_0}^1(\R^2):  H^{F}_{\alpha,B_0}f\in L^{2}(\R^2)\Big\}
\end{equation*}
Even though the operator $H_{\alpha,B_0}$ has many other self-adjoint extensions (see \cite{ESV}) by the von Neumann extension theory,  in this whole paper, we use the simplest Friedrichs extension and briefly write $H_{\alpha,B_0}$ as its Friedrichs extension $H^{F}_{\alpha, B_0}$.

\subsection{The spectrum of the operator $H_{\alpha,B_0}$}

In this subsection, we exhibit the eigenvalues and eigenfunctions of the Schr\"odinger operator
\begin{equation*}
H_{\alpha,B_0}=-(\nabla+i(A_B(x)+A_{\mathrm{hmf}}(x)))^2,
\end{equation*}
where the magnetic vector potentials are in \eqref{AB-potential} and
\eqref{A-hmf}.
\begin{proposition}[The spectrum for $H_{\alpha,B_0}$]\label{prop:spect}
Let  $H_{\alpha,B_0}$ be the self-adjoint Schr\"odinger operator in \eqref{H-A}.
Then the eigenvalues of $H_{\alpha,B_0}$ are given by
\begin{equation}\label{eigen-v}
\lambda_{k,m}=(2m+1+|k+\alpha|)B_0+(k+\alpha)B_0,\quad m,\,k\in\mathbb{Z},\, m\geq0,
\end{equation}
with (finite) multiplicity
\begin{equation*}
\#\Bigg\{j\in\mathbb{Z}:\frac{\lambda_{k,m}-(j+\alpha)B_0}{2B_0}-\frac{|j+\alpha|+1}{2}\in\mathbb{N}\Bigg\}.
\end{equation*}
Furthermore, let $\theta=\frac{x}{|x|}$, the corresponding eigenfunction is given by
\begin{equation}\label{eigen-f}
V_{k,m}(x)=|x|^{|k+\alpha|}e^{-\frac{B_0 |x|^2}{4}}\, P_{k,m}\Bigg(\frac{B_0|x|^2}{2}\Bigg)e^{ik\theta}
\end{equation}
where $P_{k,m}$ is the polynomial of degree $m$ given by
\begin{equation*}
P_{k,m}(r)=\sum_{n=0}^m\frac{(-m)_n}{(1+|k+\alpha|)_n}\frac{r^n}{n!}.
\end{equation*}
with $(a)_n$ ($a\in\R$) the Pochhammer's symbol
\begin{align*}
(a)_n=
\begin{cases}
1,&n=0;\\
a(a+1)\cdots(a+n-1),&n=1,2,\cdots
\end{cases}
\end{align*}\end{proposition}

\begin{remark} One can verify that the orthogonality holds
$$\int_{\R^2}V_{k_1,m_1}(x) V_{k_2,m_2}(x) \,dx=0,\quad \text{if}\quad (k_1, m_1)\neq (k_2, m_2).$$
\end{remark}

\begin{remark} Let $L^\alpha_m(t)$ be the generalized Laguerre polynomials
\begin{equation*}
L^\alpha_m(t)=\sum_{n=0}^m (-1)^n \Bigg(
  \begin{array}{c}
    m+\alpha \\
    m-n \\
  \end{array}
\Bigg)\frac{t^n}{n!},
\end{equation*}
then one has the well known orthogonality relation
\begin{equation*}
\int_0^\infty x^{\alpha} e^{-x}L^\alpha_m(x) L^\alpha_n (x)\, dx=\frac{\Gamma(n+\alpha+1)}{n!} \delta_{n,m},\end{equation*}
where $\delta_{n,m}$ is the Kronecker delta. Let $\tilde{r}=\frac{B_0|x|^2}{2}$ and $\alpha_k=|k+\alpha|$, then
\begin{equation}\label{P-L}
P_{k,m}(\tilde{r})=\sum_{n=0}^m\frac{(-1)^n m(m-1)\cdots(m-(n-1))}{(\alpha_k +1)(\alpha_k+2)\cdots(\alpha_k+n)}\frac{\tilde{r}^n}{n!}=\Bigg(
  \begin{array}{c}
    m+\alpha_k \\
    m \\
  \end{array}
\Bigg)^{-1}L^{\alpha_k}_m(\tilde{r}).
\end{equation}
Therefore,
\begin{equation}\label{V-km-2}
\|V_{k,m}(x)\|^2_{L^2(\R^2)} =\pi \Big(\frac{2}{B_0}\Big)^{\alpha_k+1}\Gamma(1+\alpha_k) \Bigg(
  \begin{array}{c}
    m+\alpha_k \\
    m \\
  \end{array}
\Bigg)^{-1}. \end{equation}
\end{remark}

\begin{remark}
Recall the Poisson kernel formula for Laguerre polynomials \cite[(6.2.25)]{AAR01}: for $a, b, c, \alpha>0$
\begin{equation*}
\begin{split}
&\sum_{m=0}^\infty e^{-cm}\frac{m !}{\Gamma(m+\alpha+1)} L_m^{\alpha} (a) L_m^{\alpha} (b)\\
&=\frac{e^{\frac{\alpha c}2}}{(ab)^{\frac{\alpha}2}(1-e^{-c})} \exp\left(-\frac{(a+b)e^{-c}}{1-e^{-c}}\right) I_{\alpha}\left(\frac{2\sqrt{ab}e^{-\frac{c}2}}{1-e^{-c}}\right)
\end{split}
\end{equation*}
then this together with \eqref{P-L} gives
\begin{equation}\label{Po-L}
\begin{split}
&\sum_{m=0}^\infty e^{-cm}\frac{m !}{\Gamma(m+\alpha_k+1)}
\Bigg(\begin{array}{c}
    m+\alpha_k \\
    m \\
  \end{array}
\Bigg)^2 P_{k,m} (a) P_{k,m} (b)\\
&=\frac{e^{\frac{\alpha_k c}2}}{(ab)^{\frac{\alpha_k}2}(1-e^{-c})} \exp\left(-\frac{(a+b)e^{-c}}{1-e^{-c}}\right) I_{\alpha_k}\left(\frac{2\sqrt{ab}e^{-\frac{c}2}}{1-e^{-c}}\right).
\end{split}
\end{equation}
\end{remark}
We refer to \cite{WZZ} for the proof.

\subsection{The Sobolev spaces}
In this subsection, we will prove the equivalence of two norms.

\begin{proposition}[Equivalent norms]\label{prop:E-S-B}
Let the Sobolev norm and Besov norm be defined in \eqref{Sobolev2} and \eqref{Besov} respectively. For $s\in\R$, then
there exist positive constants $c, C$ such that
\begin{equation}\label{S-B-H}
c\|f\|_{\dot \HH^s_{\alpha,B_0}(\R^2)}\leq \|f\|_{\dot{\mathcal{B}}^s_{2,2}(\R^2)}\leq C \|f\|_{\dot \HH^s_{\alpha,B_0}(\R^2)},
\end{equation}
and
\begin{equation}\label{S-B-H'}
c\|f\|_{ \HH^s_{\alpha,B_0}(\R^2)}\leq \|f\|_{{\mathcal{B}}^s_{2,2}(\R^2)}\leq C \|f\|_{ \HH^s_{\alpha,B_0}(\R^2)}.
\end{equation}
\end{proposition}

\begin{proof} Let $\tilde{V}_{k,m}$ be the $L^2$-normalization of $V_{k,m}$ in \eqref{eigen-f}, then the eigenfunctions
$\Big\{\tilde{V}_{k,m}\Big\}_{k\in\Z, m\in\mathbb{N}} $
form an orthonormal basis of $L^2(\mathbb{R}^2)$ corresponding to the eigenfunctions of $H_{\alpha,B_0}$. \vspace{0.2cm}

 By the functional calculus, for any well-behaved functions $F$ (e.g. bounded Borel measurable function) and $f\in L^2$,
we can write
\begin{equation*}
F(H_{\alpha,B_0}) f=\sum_{k\in\Z, \atop m\in\mathbb{N}} F(\lambda_{k,m}) c_{k,m}\tilde{V}_{k,m}(x).
\end{equation*}
where
\begin{equation*}
c_{k,m}=\int_{\mathbb{R}^2}f(y)\overline{\tilde{V}_{k,m}(y)} dy.
\end{equation*}
Then
\begin{equation}\label{norm-f}
\|F(H_{\alpha,B_0}) f\|_{L^2(\R^2)}=\Big(\sum_{k\in\Z, \atop m\in\mathbb{N}} \big|F(\lambda_{k,m}) c_{k,m}\big|^2\Big)^{1/2}.
\end{equation}
In particular, we have
\begin{equation*}
\|f\|_{\dot \HH^s_{\alpha,B_0}(\R^2)}=\|H^{\frac s2}_{\alpha,B_0} f\|_{L^2(\R^2)}=\Big(\sum_{k\in\Z, \atop m\in\mathbb{N}} \big|\lambda^\frac{s}2_{k,m} c_{k,m}\big|^2\Big)^{1/2}.
\end{equation*}
 Let $\varphi\in C_c^\infty(\mathbb{R}\setminus\{0\})$ in \eqref{LP-dp}. On the one hand, by the definition and \eqref{norm-f}, we have
\begin{equation*}
\begin{split}
\|f\|_{\dot{\mathcal{B}}^s_{2,2}(\R^2)}&=\Big(\sum_{j\in\Z}2^{2js}\|\varphi_j(\sqrt{H_{\alpha,B_0}})f\|_{L^2(\R^2)}^2\Big)^{1/2}\\
&=\Big(\sum_{j\in\Z}\sum_{k\in\Z, \atop m\in\mathbb{N}}  2^{2js}\big|
\varphi\Big(\frac{\sqrt{\lambda_{k,m}}}{2^j}\Big) c_{k,m}\big|^2\Big)^{1/2}\\
&\cong \Big(\sum_{j\in\Z}\sum_{k\in\Z, \atop m\in\mathbb{N}}  \lambda_{k,m}^{s}\big|
\varphi\Big(\frac{\sqrt{\lambda_{k,m}}}{2^j}\Big) c_{k,m}\big|^2\Big)^{1/2}\\
&\lesssim \Big(\sum_{k\in\Z, \atop m\in\mathbb{N}}  |\lambda_{k,m}^{\frac s2}c_{k,m}|^2\sum_{j\in\Z}\big|
\varphi\Big(\frac{\sqrt{\lambda_{k,m}}}{2^j}\Big) \big|^2\Big)^{1/2}\\
&\lesssim \Big(\sum_{k\in\Z, \atop m\in\mathbb{N}}  |\lambda_{k,m}^{\frac s2}c_{k,m}|^2\Big)^{1/2}=\|f\|_{\dot \HH^s_{\alpha,B_0}(\R^2)}.
\end{split}
\end{equation*}
On the other hand,  we have
\begin{equation*}
\begin{split}
\|f\|_{\dot \HH^s_{\alpha,B_0}(\R^2)}&=\Big(\sum_{k\in\Z, \atop m\in\mathbb{N}}  |\lambda_{k,m}^{\frac s2}c_{k,m}|^2\Big)^{1/2}\\
&=\Big(\sum_{k\in\Z, \atop m\in\mathbb{N}}  \Big|\sum_{j\in\Z} \varphi\Big(\frac{\sqrt{\lambda_{k,m}}}{2^j}\Big)\lambda_{k,m}^{\frac s2}c_{k,m}\Big|^2\Big)^{1/2}\\
&\leq C\Big(\sum_{j\in\Z}\sum_{k\in\Z, \atop m\in\mathbb{N}}  2^{2js}\big|
\varphi\Big(\frac{\sqrt{\lambda_{k,m}}}{2^j}\Big) c_{k,m}\big|^2\Big)^{1/2}\\
&\lesssim\Big(\sum_{j\in\Z}2^{2js}\|\varphi_j(\sqrt{H_{\alpha,B_0}})f\|_{L^2(\R^2)}^2\Big)^{1/2}=\|f\|_{\dot{\mathcal{B}}^s_{2,2}(\R^2)}.
\end{split}
\end{equation*}
In the above inequality, we have used the fact that, for a fixed $\lambda$, there are only finite terms in the summation
$$1=\sum_{j\in\Z}\varphi\big(\frac{\lambda}{2^j}\big).$$
Above all, we have proved \eqref{S-B-H}. One can prove \eqref{S-B-H'} similarly.

\end{proof}

\section{Heat kernel estimates}\label{sec:heat}

In this section, for our purpose of the Littlewood-Paley theory associated with $H_{\alpha,B_0}$, we study the heat kernel estimates associated with the magnetic Schr\"odinger operator $H_{\alpha,B_0}$.
We provide two methods to study the heat kernel. In the first method,
we first combine the strategies of \cite{FFFP1,GYZZ22,FZZ22} to construct the heat kernel by using the spectrum property in Proposition \ref{prop:spect}.
And then we use the representation of the heat kernel to obtain the on-diagonal estimates. Finally we extend the on-diagonal bounds by
adding the Gaussian factor $\exp(-d^2(x,y)/Ct)$ to obtain the off-diagonal Gaussian bounds. In the second one, we directly construct the heat kernel by using the Schulman-Sunada formula in \cite{stov,stov1} and then optimal the established bounds.
 \vspace{0.2cm}

\subsection{Method I: }More precisely, we will first prove the following result.

\begin{proposition}\label{prop:H}
Let $H_{\alpha,B_0}$ be the operator in \eqref{H-A} and suppose $x=r_1(\cos\theta_1,\sin\theta_1)$
and $y=r_2(\cos\theta_2,\sin\theta_2)$. Let $u(t,x)$ be the solution of the heat equation
\begin{equation}\label{equ:H}
\begin{cases}
\big(\partial_t+H_{\alpha,B_0}\big)u(t,x)=0,\\
u(0,x)=f(x).
\end{cases}
\end{equation}
Then
\begin{equation*}
u(t,x)=e^{-tH_{\alpha,B_0}}f=\int_{\R^2} K_H(t; x,y) f(y)\, dy,\quad t>0,
\end{equation*}
where the kernel of the heat propagator $e^{-tH_{\alpha,B_0}}$ is given by
\begin{equation}\label{rep:h}
K_H(t; x,y)=\Big(\frac{B_0e^{-\alpha tB_0}}{4\pi^2\sinh tB_0}\Big) e^{-\frac{B_0(r_1^2+r_2^2)\cosh tB_0}{4\sinh tB_0}}\sum_{k\in\Z}e^{ik(\theta_1-\theta_2+itB_0)}I_{\alpha_k}\left(\frac{B_0r_1r_2}{2\sinh tB_0}\right).
\end{equation}
Furthermore, there exists a constant $C$ such that
\begin{equation}\label{est:H-radial}
|K_H(t; x,y)|\leq C\frac{B_0e^{(1-\alpha)B_0 t}}{\sinh tB_0}e^{-\frac{B_0(r_1-r_2)^2}{4\tanh tB_0}}.
\end{equation}
\end{proposition}

\begin{remark} The argument is a bit different from the proof for the Schr\"odinger propagator. In particular, at first glance, the factor $e^{-tB_0k}$ is a trouble
in the summation of the formula \eqref{rep:h} when $k\in -\mathbb{N}$, but it converges due to
the factor $\sinh(tB_0)$ in the modified Bessel function.
\end{remark}

%\begin{remark}Note that $|x-y|\geq |r_1-r_2|$, and $|x-y|=|r_1-r_2|$ if $\theta_1=\theta_2$, the estimate \eqref{est:H-radial} is weaker than
%\begin{equation}\label{est:H}
%|K_H(t; x,y)|\leq C\frac{B_0e^{(1-\alpha)B_0 t}}{|\sinh tB_0|}e^{-\frac{B_0|x-y|^2}{4|\tanh tB_0|}},
%\end{equation}
%which is an open problem.
%However, from \eqref{est:H-radial}, it is easy to see the on-diagonal estimate
%\begin{equation}\label{est:H-on-dia}
%|K_H(t; x,x)|\leq C\frac{B_0e^{(1-\alpha)B_0 t}}{|\sinh tB_0|}.
%\end{equation}
%
%\end{remark}

\begin{proof}
We construct the representation formula \eqref{rep:h} of the heat flow $e^{-tH_{\alpha,B_0}}$ by combining the argument of \cite{FFFP1} and \cite{FZZ22, GYZZ22}.
This is close to the construction of Schr\"odinger flow in our previous paper \cite{WZZ}, however, we provide the details again for self-contained.
\vspace{0.2cm}

Our starting point is the Proposition \ref{prop:spect}. Let $\tilde{V}_{k,m}$ be the $L^2$-normalization of $V_{k,m}$ in \eqref{eigen-f}, then the eigenfunctions
$\Big\{\tilde{V}_{k,m}\Big\}_{k\in\Z, m\in\mathbb{N}} $
form an orthonormal basis of $L^2(\mathbb{R}^2)$ corresponding to the eigenfunctions of $H_{\alpha,B_0}$. \vspace{0.2cm}

We expand the initial data $f(x)\in L^2$ as
\begin{equation*}
f(x)=\sum_{k\in\Z, \atop m\in\mathbb{N}} c_{k,m}\tilde{V}_{k,m}(x)
\end{equation*}
where
\begin{equation}\label{cmk1}
c_{k,m}=\int_{\mathbb{R}^2}f(x)\overline{\tilde{V}_{k,m}(x)}\, dx.
\end{equation}
The solution $u(t,x)$ of \eqref{equ:H} can be written as
\begin{equation}\label{sol:expan}
u(t,x)=\sum_{k\in\Z, \atop m\in\mathbb{N}} u_{k,m}(t)\tilde{V}_{k,m}(x),
\end{equation}
where $u_{k,m}(t)$ satisfies the ODE
\begin{equation*}
\begin{cases}
&u'_{k,m}(t)=-\lambda_{k,m}u_{k,m}(t),\\
& u_{k,m}(0)=c_{k,m},\quad k\in\Z, \, m\in\mathbb{N}.
\end{cases}
\end{equation*}
Thus we obtain $u_{k,m}(t)=c_{k,m}e^{-t\lambda_{k,m}}$.
Therefore the solution \eqref{sol:expan} becomes
\begin{equation*}
u(t,x)=\sum_{k\in\Z, \atop m\in\mathbb{N}} c_{k,m}e^{-t\lambda_{k,m}}\tilde{V}_{k,m}(x).
\end{equation*}
Plugging \eqref{cmk1} into the above expression yields
\begin{equation*}
u(t,x)=\sum_{k\in\Z, \atop m\in\mathbb{N}} e^{-t\lambda_{k,m}}\left(\int_{\mathbb{R}^2}f(y)\overline{\tilde{V}_{k,m}(y)} dy\right)\tilde{V}_{k,m}(x).
\end{equation*}
We write $f$ in a harmonic spherical expansion
\begin{equation*}
f(y)=\sum_{k\in\Z} f_k(r_2) e^{ik\theta_2},
\end{equation*}
where
\begin{equation}\label{f-k}
f_k(r_2)=\frac{1}{2\pi}\int_0^{2\pi} f(r_2, \theta_2) e^{-ik\theta_2}\, d\theta_2,\quad r_2=|y|,
\end{equation}
we thus have
\begin{align*}
u(t,x)&=\sum_{k\in\Z, \atop m\in\mathbb{N}} e^{-t\lambda_{k,m}}\frac{V_{k,m}(x)}{\|V_{k,m}\|^2_{L^2}}\Bigg(\int_0^\infty f_k(r_2)e^{-\frac{B_0r_2^2}{4}}\, P_{k,m}\Big(\frac{B_0r_2^2}{2}\Big) r_2^{1+\alpha_k}\mathrm{d}r_2\Bigg)\\
&=\Big(\frac{B_0}{2\pi}\Big)\sum_{k\in\Z}e^{ik\theta_1}\frac{B_0^{\alpha_k}e^{-t\beta_k}}{2^{\alpha_k}\Gamma(1+\alpha_k)}\Bigg[\sum_{m=0}^\infty
\Bigg(
  \begin{array}{c}
    m+\alpha_k \\
    m \\
  \end{array}
\Bigg)e^{-2tmB_0}\\
&\times\Bigg(\int_0^\infty f_k(r_2)(r_1r_2)^{\alpha_k}e^{-\frac{B_0(r_1^2+r_2^2)}{4}}P_{k,m}\left(\frac{B_0r_2^2}{2}\right)P_{k,m}\left(\frac{B_0r_1^2}{2}\right)r_2 \mathrm{d}r_2\Bigg)\Bigg],
\end{align*}
where $\alpha_k=|k+\alpha|$ and we use \eqref{eigen-v},\eqref{eigen-f},\eqref{V-km-2} and
\begin{equation*}
\begin{split}
\lambda_{k,m}&=(2m+1+|k+\alpha|)B_0+(k+\alpha)B_0\\
&:=2mB_0+\beta_k
\end{split}
\end{equation*}
with $\beta_k=(1+|k+\alpha|)B_0+(k+\alpha)B_0\geq B_0>0$.\vspace{0.2cm}

Using the formula \eqref{Po-L} and \eqref{f-k}, we obtain
\begin{align*}
u(t,x)&=\Big(\frac{B_0}{4\pi^2}\Big)\int_0^\infty\int_0^{2\pi}\sum_{k\in\Z}e^{ik(\theta_1-\theta_2)}\frac{B_0^{\alpha_k}e^{-t\beta_k}}{2^{\alpha_k}\Gamma(1+\alpha_k)}
 (r_1r_2)^{\alpha_k}e^{-\frac{B_0(r_1^2+r_2^2)}{4}}f(r_2,\theta_2)\\
&\times \frac{2^{\alpha_k}e^{t\alpha_k B_0}}{(B_0r_1r_2)^{\alpha_k}}\exp\left(-\frac{\frac{B_0(r_1^2+r_2^2)}2e^{-2tB_0}}{1-e^{-2tB_0}}\right)I_{\alpha_k}
\left(\frac{B_0r_1r_2e^{-tB_0}}{1-e^{-2tB_0}}\right)r_2 \mathrm{d}r_2\mathrm{d}\theta_2\\
&=\Big(\frac{B_0e^{-\alpha tB_0}}{4\pi^2\sinh tB_0}\Big)\int_0^\infty\int_0^{2\pi} e^{-\frac{B_0(r_1^2+r_2^2)\cosh tB_0}{4\sinh tB_0}}f(r_2,\theta_2)\\
&\quad\times\sum_{k\in\Z}e^{ik(\theta_1-\theta_2+itB_0)}I_{\alpha_k}\left(\frac{B_0r_1r_2}{2\sinh tB_0}\right)r_2 \mathrm{d}r_2\mathrm{d}\theta_2.
\end{align*}
Therefore, we obtain the heat kernel
\begin{equation*}
K_H(t; x,y)=\Big(\frac{B_0e^{-\alpha tB_0}}{4\pi^2\sinh tB_0}\Big) e^{-\frac{B_0(r_1^2+r_2^2)\cosh tB_0}{4\sinh tB_0}}\sum_{k\in\Z}e^{ik(\theta_1-\theta_2+itB_0)}I_{\alpha_k}\left(\frac{B_0r_1r_2}{2\sinh tB_0}\right),
\end{equation*}
which gives \eqref{rep:h}. \vspace{0.2cm}

Now we need to verify the inequality \eqref{est:H-radial}. To this end, it suffices to show
\begin{equation}\label{H-Up}
|K_H(t; x,y)|\leq\frac{B_0e^{-\alpha B_0t}}{4\pi^2\sinh tB_0}e^{-\frac{B_0(r_1^2+r_2^2)}{4\tanh tB_0}}\sum_{k\in\Z}e^{-kB_0 t}I_{\alpha_k}\left(\frac{B_0r_1r_2}{2\sinh tB_0}\right).
\end{equation}
Let $z=\frac{B_0r_1r_2}{2\sinh tB_0}>0$ and notice the monotonicity of the modified Bessel function $I_{\mu}(z)$ with respect to the order, in other words,  for fixed $z>0$, $$ I_{\mu}(z)\leq I_{\nu}(z), \quad \mu\geq \nu.$$
Recall $\alpha_k=|k+\alpha|$ and $\alpha\in (0,1)$,  thus we show
\begin{align*}
\sum_{k\in\Z}e^{-kB_0 t}I_{\alpha_k}\left(z\right)
=&\sum_{k\geq0}e^{-kB_0 t}I_{k+\alpha}(z)+\sum_{k\geq1}e^{kB_0 t}I_{k-\alpha}(z)\\
\leq &\sum_{k\geq0}e^{-kB_0 t}I_{k}(z)+e^{B_0 t}\sum_{k\geq0}e^{kB_0 t}I_{k}(z)\\
\leq &e^{B_0 t}\Big(\sum_{k\geq0}e^{-kB_0 t}I_{k}(z)+\sum_{k\geq0}e^{kB_0 t}I_{k}(z)\Big)\\
\leq &e^{B_0 t}\Big(\sum_{k\in\Z}e^{kB_0 t}I_{|k|}(z)+I_{0}(z)\Big)\\
\leq &Ce^{B_0t}(e^z+e^{z\cosh(B_0 t)})\\
\leq &Ce^{B_0 t}e^{z\cosh(B_0 t)}\\
\leq &Ce^{B_0 t}e^{\frac{B_0r_1r_2}{2|\tanh tB_0|}},
\end{align*}
where we use the formula \cite[Eq. (9.6.19)]{AS65}
\begin{equation*}
\sum_{k\in\Z} e^{kt}I_{|k|}(z)=e^{z\cosh(t)},\quad I_0(z)\leq C e^{z}.
\end{equation*}
Combining with \eqref{H-Up}, we have verified \eqref{est:H-radial}.
\end{proof}

We next extend our result of the "on-diagonal" 
kernel estimate
\begin{equation*}
|K_H(t; x,x)|\leq C\frac{B_0e^{(1-\alpha)B_0 t}}{|\sinh tB_0|}
\end{equation*}
to the "off-diagonal".  Let $p_t(x, y)$ denote the heat kernel corresponding to a second-order differential elliptic or sub-elliptic operator,
then the usual theory says that one can automatically improve on-diagonal bounds
$$p_t(x,x)\leq \frac{C}{V(x,\sqrt{t})}$$ to
the typical Gaussian heat kernel upper bound
$$p_t(x,y)\leq \frac{C}{V(x,\sqrt{t})}\exp\Big(-\frac{d^2(x,y)}{Ct}\Big)$$
for all $t>0$ and $x, y$ ranging in the space where the operator acts, for an appropriate function $V$.\vspace{0.2cm}

For our specific operator $H_{\alpha,B_0}$, we prove that
\begin{proposition}\label{prop:Gassian-H}
Let $K_H(t; x,y)$ be in Proposition \ref{prop:H}, then there exists a constant $C$ such that
\begin{equation*}
|K_H(z; x,y)|\leq C (\Re z)^{-1} \exp\Big(-\Re\frac{d^2(x,y)}{Cz}\Big),
\end{equation*}
for all $z\in \C_+$ and $x, y\in \R^2$. In particular, $z=t>0$, then
\begin{equation}\label{est:G-H}
|K_H(t; x,y)|\leq C t^{-1} \exp\Big(-\frac{d^2(x,y)}{Ct}\Big),\quad t>0.
\end{equation}
\end{proposition}

\begin{remark} One usual way to prove the Gaussian bounds for the magnetic Schr\"odinger operator is to apply the important diamagnetic
inequality
\begin{equation}\label{diamagnetic-h}
\Big|\Big(e^{t(\nabla+iA(x))^2}f\Big)(x)\Big|\leq \Big(e^{t\Delta}|f|\Big)(x),
\end{equation}
which relates estimates on the magnetic Schr\"odinger operator semigroup to estimates
on the free heat semigroup. The obvious disadvantage of using \eqref{diamagnetic-h} is that all the effects of the magnetic field
are completely eliminated. To our best knowledge, \eqref{diamagnetic-h} is available for $A(x)\in L^2_{\mathrm{loc}}$, see \cite{Simon}. Unfortunately, our magnetic potential $A(x)=A_B(x)+A_{\mathrm{hmf}}(x)\notin L^2_{\mathrm{loc}}(\R^2)$.
\end{remark}

\begin{remark} To recover some magnetic effects, it would be tempting to prove
\begin{equation}\label{diamagnetic-h1}
\Big|\Big(e^{-tH_{\alpha,B_0}}f\Big)(x)\Big|\leq \Big(e^{-tH_{0,B_0}}|f|\Big)(x),
\end{equation}
or
\begin{equation}\label{diamagnetic-h2}
\Big|\Big(e^{-tH_{\alpha,B_0}}f\Big)(x)\Big|\leq \Big(e^{-tH_{\alpha, 0}}|f|\Big)(x).
\end{equation}
If \eqref{diamagnetic-h1} was available,
\begin{equation}\label{diamagnetic-h1-g}
\big|e^{-tH_{\alpha,B_0}}(x,y)\big|\lesssim \big|e^{-tH_{0,B_0}}(x,y)\big|\lesssim \frac{1}{\sinh(B_0t)}e^{-\frac{B_0|x-y|^2}{4\tanh(B_0t)}},
\end{equation}
where we use the Mehler heat kernel of $e^{-tH_{0,B_0}}$ (e.g. \cite[P168]{Simon2})
\begin{equation}\label{equ:Meh}
e^{-tH_{0,B_0}}(x,y)=\frac{B_0}{4\pi\sinh(B_0t)}e^{-\frac{B_0|x-y|^2}{4\tanh(B_0t)}-\frac{iB_0}2(x_1y_2-x_2y_1)}.
\end{equation}
If \eqref{diamagnetic-h2} was available, we obtain
\begin{equation}\label{diamagnetic-h2-b}
\big|e^{-tH_{\alpha,B_0}}(x,y)\big|\lesssim \big|e^{-tH_{\alpha,0}}(x,y)\big|\lesssim  \frac{1}{t}e^{-\frac{|x-y|^2}{4t}}.
\end{equation}
We refer to \cite[Proposition 3.2]{FZZ22} for the last  Gaussian upper bounds for  $e^{-tH_{\alpha, 0}}(x,y)$.
Since for $t\geq 0$, one has
$$\sinh(t)\geq t, \quad t/\tanh(t)\geq 1,$$  hence \eqref{diamagnetic-h2-b} is weaker than \eqref{diamagnetic-h1-g}. Unfortunately,
as pointed out in \cite{LT},  the semigroup generated by the magnetic Schr\"odinger operator  is not Markovian, in fact not even positivity
preserving which is important in the theory of comparison of heat semigroups. So the truth of \eqref{diamagnetic-h1} and \eqref{diamagnetic-h2} is not known, we refer to \cite{LT}.
\end{remark}

\begin{remark} The Ganssian decay on the right side of \eqref{est:G-H} is the one of the heat kernel, which
is considerably weaker than the decay of the Mehler kernel \eqref{diamagnetic-h1-g}. Similarly as \cite{LT}, we ask that how to prove
\begin{equation*}
|K_H(t; x,y)|\lesssim \frac{B_0e^{-\alpha B_0 t}}{|\sinh tB_0|}e^{-\frac{|x-y|^2}{C\tanh(B_0t)}},\quad t>0.
\end{equation*}
The truth of this estimate would reveal a
robust dependence of the magnetic heat kernel on the magnetic field. In our case, we give a positive answer to this problem by proving \eqref{est:heatkernel:2}
in the subsequent subsection.

\end{remark}

%\begin{remark} Let $\Delta \theta=\theta_1-\theta_2$, notice that $$|x-y|^2=r_1^2+r_2^2-2r_1r_2\cos(\theta_1-\theta_2)=(r_1-r_2)^2+4r_1r_2\sin^2\big(\frac{\Delta\theta}2\big).$$
%If $r_1\geq 2r_2$ or $r_1\leq \frac{r_2}2$, then there exists a large constant $C$ such that $(r_1-r_2)^2\geq |x-y|^2/C$. Therefore, \eqref{est:G-H} is a direct consequence of
%\eqref{est:H-radial} due to the fact that $t\coth(t)\geq 1$. Hence it suffices to prove \eqref{est:G-H} under the condition
%\begin{equation}
%r_1\sim r_2,\quad \sin^2\big(\frac{\Delta\theta}2\big)\gtrsim \frac{(r_1-r_2)^2}{r_1r_2}.
%\end{equation}
%
%\end{remark}

\begin{proof} We prove \eqref{est:G-H} by using \cite[Theorem 4.2]{CS08}. The Theorem claims that
if $(M, d, \mu, L)$ satisfies the Davies-Gaffney estimates, that is,
\begin{equation*}
|\langle e^{-tL}f, g\rangle |\leq \|f\|_{L^2(U_1)}\|g\|_{L^2(U_2)} e^{-\frac{d^2(U_1,U_2)}{4t}}
\end{equation*}
for all $t>0$, $U_i\subset M$ with $i=1,2$ and $f\in L^2(U_1, d\mu)$, $g\in L^2(U_2, d\mu)$ and
 $d(U_1,U_2)=\inf\{\rho=|x-y|: x\in U_1, y\in U_2\}$.
If, for some $K$ and $D>0$,
\begin{equation*}
e^{-tL}(x,x)\leq K t^{-\frac{D}2},\qquad \forall t>0, \quad x\in M,
\end{equation*}
then
\begin{equation*}
|e^{-zL}(x,y)|\leq K (\Re z)^{-\frac{D}2}\Big(1+\Re\frac{d^2(x,y)}{4z}\Big)^{\frac D2}\exp\Big(-\Re\frac{d^2(x,y)}{4z}\Big)
\end{equation*}
for all $z\in \C_+$ and $x, y\in M$.\vspace{0.2cm}

For our model $M=\R^2$ and $L=H_{\alpha,B_0}$, we need to verify the on-diagonal estimates
\begin{equation}\label{est:on-dia}
e^{-tH_{\alpha,B_0}}(x,x)\leq K t^{-1},\qquad \forall t>0, \quad x\in M,
\end{equation}
and the Davies-Gaffney estimates
\begin{equation}\label{est:DG-H}
|\langle e^{-tH_{\alpha,B_0}}f, g\rangle |\leq \|f\|_{L^2(U_1)}\|g\|_{L^2(U_2)} e^{-\frac{d^2(U_1,U_2)}{4t}}.
\end{equation}
If this has been done, for $z=t>0$ and $D=2$ and $\epsilon>0$, then
\begin{equation*}
\begin{split}
|e^{-tH_{\alpha,B_0}}(x,y)|&\leq Ct^{-1}\Big(1+\frac{|x-y|^2}{4t}\Big)\exp\Big(-\frac{|x-y|^2}{4t}\Big)\\
&\leq Ct^{-1}\exp\Big(-\frac{|x-y|^2}{(4+\epsilon)t}\Big).
\end{split}
\end{equation*}
Therefore, it suffices to verify \eqref{est:on-dia} and \eqref{est:DG-H}. Since $x=y$ and $\frac{B_0e^{-\alpha B_0t}}{4\pi^2|\sinh tB_0|}\leq C t^{-1}$, the estimate \eqref{est:on-dia} is a direct consequence of \eqref{est:H-radial}.
However, the inequality \eqref{est:DG-H} is more complicated, this is a consequence of  \eqref{est:DG} below.

\end{proof}

\begin{proposition}[Davies-Gaffney inequality] Let $A$ and $B$ be two disjoint measurable sets in $\R^2$ and suppose that $f\in L^2(A)$ and $g\in L^2(B)$ such that $\mathrm{supp}(f)\subset A$ and $\mathrm{supp}(g)\subset B$. Then
\begin{equation}\label{est:DG}
|\langle e^{-tH_{\alpha,B_0}}f, g\rangle |\leq \|f\|_{L^2(A)}\|g\|_{L^2(B)} e^{-\frac{d^2(A,B)}{4t}}
\end{equation}
where $d(A,B)=\inf\{\rho=|x-y|: x\in A, y\in B\}$.
\end{proposition}

\begin{proof} Let $\rho=d(A,B)$ and define
$$A_{\rho}=\{x\in\R^2: d(x, A)<\rho\}, \quad A^c_{\rho}=\R^2\setminus A_{\rho}$$
where $d(x,A)=\inf\{|x-y|: y\in A\}$.
Then $B\subset A^c_{\rho}$, furthermore, by Cauchy-Schwartz inequality, we have
\begin{equation*}
\begin{split}
|\langle e^{-tH_{\alpha,B_0}}f, g\rangle |&\leq \Big(\int_{B} |e^{-tH_{\alpha,B_0}}f|^2 dx\Big)^{1/2} \|g\|_{L^2(B)}\\
&\leq \Big(\int_{A^c_{\rho}} |e^{-tH_{\alpha,B_0}}f|^2 dx\Big)^{1/2} \|g\|_{L^2(B)}.
\end{split}
\end{equation*}
Therefore, \eqref{est:DG} follows if we could prove
\begin{equation}\label{est:DG1}
\begin{split}
\int_{A^c_{\rho}} |e^{-tH_{\alpha,B_0}}f|^2 dx\leq  \|f\|^2_{L^2(A)} e^{-\frac{d^2(A,B)}{2t}}.
\end{split}
\end{equation}
To this end, for any fixed $s>t$ and $x\in\R^2$ and $\tau\in [0,s)$, we define the function
\begin{equation*}
\xi(\tau, x):=\frac{d^2(x, A^c_{\rho})}{2(\tau-s)},
\end{equation*}
and set
\begin{equation}\label{J:funct}
J(\tau):=\int_{\R^2} \big|e^{-\tau H_{\alpha,B_0}}f\big|^2 e^{\xi(\tau, x)}\, dx.
\end{equation}

\begin{lemma}\label{lem: J} For the function defined in \eqref{J:funct}, we have that
\begin{equation}\label{est:J}
J(t)\leq J(0).
\end{equation}
\end{lemma}
We assume \eqref{est:J} to prove \eqref{est:DG1} by postponing the proof for a moment.
Since $x\in A^c_{\rho}$, one has $\xi(\tau, x)=0$, thus
\begin{equation}\label{est:DG2}
\begin{split}
\int_{A^c_{\rho}} |e^{-tH_{\alpha,B_0}}f|^2 dx\leq  \int_{\R^2} \big|e^{-t H_{\alpha,B_0}}f\big|^2 e^{\xi(t, x)}\, dx=J(t).
\end{split}
\end{equation}
For $t=0$, since
\begin{equation*}
e^{\xi(0,x)}\leq
\begin{cases}
 1,&\qquad x\in A^c_{\rho};\\
e^{-\frac{d^2(x, A^c_{\rho})}{2s}},&\qquad x\in A,
\end{cases}
\end{equation*} we see
\begin{equation}\label{J0}
J(0)=\int_{\R^2} f(x)^2 e^{\xi(0, x)}\, dx\leq  \int_{A^c_{\rho}} |f(x)|^2 \, dx+\exp\Big(-\frac{\rho^2}{2s}\Big)\int_{A} |f(x)|^2 \, dx.
\end{equation}
By using \eqref{est:DG2}, \eqref{est:J} and \eqref{J0} and taking $s\to t+$, we obtain
\begin{equation*}
\begin{split}
&\int_{A^c_{\rho}} |e^{-tH_{\alpha,B_0}}f|^2 dx\leq J(t)\leq J(0)\\
&\leq C\left( \int_{A^c_{\rho}} |f(x)|^2 \, dx+\exp\Big(-\frac{\rho^2}{2s}\Big)\int_{A} |f(x)|^2 \, dx\right)\\
&\leq  \int_{A^c_{\rho}} |f(x)|^2 \, dx+\exp\Big(-\frac{\rho^2}{2t}\Big)\|f\|^2_{L^2(A)}.
\end{split}
\end{equation*}
Since $f\in L^2(A)$ and $\mathrm{supp}(f)\subset A$, then $\int_{A^c_{\rho}} |f(x)|^2 \, dx=0$ which implies \eqref{est:DG1}.

\end{proof}
Now it remains to prove \eqref{est:J} in Lemma \ref{lem: J}.
\begin{proof}[The proof of Lemma \ref{lem: J}] Indeed, we need to prove that the function $J(\tau)$ defined in \eqref{J:funct} is non-increasing in $\tau\in [0, s)$. We closely follow the argument of  the
integrated maximum principle \cite[Theorem 12.1]{Grig}. Furthermore, for all
$\tau, \tau_0 \in [0, s)$, if $\tau>\tau_0$, then
\begin{equation}\label{est:J'}
J(\tau)\leq J(\tau_0).
\end{equation}
which shows \eqref{est:J} by taking $\tau=t$ and $\tau_0=0$. Without loss of generality, we assume $f\geq 0$ in \eqref{J:funct}. Indeed, if $f$ has a change sign, we set
$g= |e^{-\tau_0 H_{\alpha,B_0}}f|\ge 0$, then
$$ |e^{-\tau H_{\alpha,B_0}}f|= |e^{-(\tau -\tau_0)H_{\alpha,B_0}}e^{ -\tau_0H_{\alpha,B_0}}f|\leq e^{-(\tau -\tau_0)H_{\alpha,B_0}} g.$$
Assume that \eqref{est:J'} holds for $g\geq 0$, then
\begin{equation*}
\begin{split}
J(\tau)&=\int_{\R^2} \big(e^{-\tau H_{\alpha,B_0}}f\big)^2 e^{\xi(\tau, x)}\, dx\\
&\leq \int_{\R^2} \big(e^{-(\tau-\tau_0) H_{\alpha,B_0}}g\big)^2 e^{\xi(\tau, x)}\, dx\\
&\leq \int_{\R^2} g^2(x) e^{\xi(\tau_0, x)} dx\\&
= \int_{\R^2} \big(e^{ -\tau_0H_{\alpha,B_0}}f\big)^2 e^{\xi(\tau_0, x)} dx\\
&= J(\tau_0).
\end{split}
\end{equation*}
From now on, we assume $f\geq 0$. By using \cite[Theorem 5.23]{Grig}(which claims $e^{-\tau H^{\Omega_i}_{\alpha,B_0}}f\to e^{-\tau H_{\alpha,B_0}}f$ in $L^2$ as $i\to+\infty$ where $\Omega_0\subset\Omega_1\subset\cdots\subset\Omega_i\subset\cdots\to \R^2$ as $i\to+\infty$ ), it suffices to show that, for any relatively compact open set $\Omega\subset \R^2$,
the function
\begin{equation*}
J_\Omega (\tau):=\int_{\Omega} \big|e^{-\tau H^{\Omega}_{\alpha,B_0}}f\big|^2 e^{\xi(\tau, x)}\, dx
\end{equation*}
 is non-increasing in $\tau\in [0, s)$, where $H^{\Omega}_{\alpha,B_0}$ is the Dirichelet Laplace operator in $\Omega$
 $$H^{\Omega}_{\alpha,B_0}=H_{\alpha,B_0}\big|_{W^2_0(\Omega)\cap D(H_{\alpha,B_0})}.$$ To this end, we need to prove the derivative on $J_\Omega (\tau)$ w.r.t. $\tau$ is non-positive.

 By using \cite[Theorem 4.9]{Grig}, the function $u(t,\cdot)=e^{-t H^{\Omega}_{\alpha,B_0}}f$ is strongly differentiable in $L^2(\Omega)$ and its strong derivative
$\frac{du}{dt}$ in $L^2(\Omega)$ is given by
\begin{equation*}
\frac{du}{dt}=-H^{\Omega}_{\alpha,B_0} u.
\end{equation*}
Then we have
 \begin{equation}\label{d-J}
\begin{split}
\frac{d J_\Omega(\tau)}{d\tau}&=\Re \frac{d}{d\tau}\langle u, ue^{\xi(\tau,x)}\rangle\\
&=\Re\langle \frac{d u}{d\tau}, ue^{\xi(\tau,x)}\rangle+ \Re \langle u, \frac{ d(ue^{\xi(\tau,x)})}{d\tau}\rangle\\
&=2\Re\langle \frac{d u}{d\tau}, ue^{\xi(\tau,x)}\rangle+ \langle |u|^2, \frac{ d(e^{\xi(\tau,x)})}{d\tau}\rangle\\
&=2\Re\langle -H^{\Omega}_{\alpha,B_0} u, ue^{\xi(\tau,x)}\rangle+ \langle |u|^2, \frac{ d(e^{\xi(\tau,x)})}{d\tau}\rangle.
\end{split}
\end{equation}
Since $e^{\xi(\tau,\cdot)}\in \mathrm{Lip}_{\mathrm{loc}}(\R^2)$, one has $e^{\xi(\tau,\cdot)}\in \mathrm{Lip}(\Omega)$. The solution $u(t, \cdot)\in W^1_0(\Omega)$,
hence $e^{\xi(\tau,\cdot)}u(t, \cdot)\in W^1_0(\Omega)$.
On the one hand, recall the operator
 \begin{equation*}
H^{\Omega}_{\alpha,B_0} =- (\nabla+i(A_B(x)+A_{\mathrm{hmf}}(x)))^2,
\end{equation*}
by using the Green formula, we obtain
\begin{align}\label{d-J1}
&2\Re \langle -H^{\Omega}_{\alpha,B_0} u, ue^{\xi(\tau,x)}\rangle=2\langle (\nabla+i(A_B(x)+A_{\mathrm{hmf}}(x)))^2 u, ue^{\xi(\tau,x)}\rangle\\
&=-2\int_{\Omega} |(\nabla+i(A_B(x)+A_{\mathrm{hmf}}(x))) u|^2 e^{\xi(\tau,x)}\,dx-2\Re \int_{\Omega} \nabla u\cdot\nabla \xi e^{\xi(\tau,x)}\bar{u}\, dx. \nonumber
\end{align}
On the other hand, we observe that
$$\frac{ d(e^{\xi(\tau,x)})}{dt}=e^{\xi(\tau,x)}\frac{\partial \xi}{\partial\tau},$$
then
 \begin{equation}\label{d-J2}
\begin{split}
 \langle |u|^2, \frac{ d(e^{\xi(\tau,x)})}{dt}\rangle&= \int_{\Omega} |u|^2 e^{\xi(\tau,x)}\frac{\partial \xi(\tau,x)}{\partial\tau} \,dx\\
 &\leq
-\frac12 \int_{\Omega} |u|^2 e^{\xi(\tau,x)}|\nabla \xi|^2 \,dx.
\end{split}
\end{equation}
This is because that we can verify that the function $\xi(\tau, x)$ satisfies
$$\frac{\partial \xi}{\partial \tau}+\frac12|\nabla \xi|^2=-\frac{d^2(x, A^c_{\rho})}{2(\tau-s)^2}
+\frac12 \Bigg(\frac{d(x, A^c_{\rho})|\nabla d (x, A^c_{\rho})|}{2(\tau-s)}\Bigg)^2\leq -\frac34\frac{d^2(x, A^c_{\rho})}{2(\tau-s)^2}\leq 0,$$
since $\|\nabla f\|_{L^\infty}\leq \|f\|_{\mathrm{Lip}}$ and the function $x\mapsto d(x,E)$ is Lipschitz function with Lipschitz norm $1$, see \cite[Lemma 11.2 and Theorem 11.3]{Grig}.
Therefore, by collecting \eqref{d-J}, \eqref{d-J1} and \eqref{d-J2}, we finally show
\begin{equation*}
\begin{split}
\frac{d J_\Omega(\tau)}{d\tau}&\leq -2\int_{\Omega} |(\nabla+i(A_B(x)+A_{\mathrm{hmf}}(x))) u|^2 e^{\xi(\tau,x)}\,dx\\
&\quad -2\int_{\Omega} \Big(\Re \big(\nabla u\cdot\nabla \xi \bar{u} \big)+\frac14 |u|^2 |\nabla \xi|^2 \Big) e^{\xi(\tau,x)}\,dx.
\end{split}
\end{equation*}
On the one hand, we notice that
 $$\Re(\nabla u\cdot\nabla\xi \bar{u})=\frac{1}2\big((\nabla u\cdot\nabla\xi \bar{u})+(\nabla \bar{u}\cdot\nabla\xi u)\big)=\frac{1}2\big(\nabla |u|^2\cdot\nabla\xi \big)=(\nabla |u|) \cdot\nabla\xi |u|.$$
Therefore, we have
 \begin{equation*}
\begin{split}
\frac{d J_\Omega(\tau)}{d\tau}&\leq -2\int_{\Omega} \big(|(\nabla+i(A_B(x)+A_{\mathrm{hmf}}(x))) u|^2-|\nabla |u||^2\big) e^{\xi(\tau,x)}\,dx\\
&\quad -2\int_{\Omega} \Big(|\nabla |u||^2+(\nabla |u|) \cdot\nabla\xi |u|+\frac14 |u|^2 |\nabla \xi|^2 \Big) e^{\xi(\tau,x)}\,dx\\
&=-2\int_{\Omega} \big(|(\nabla+i(A_B(x)+A_{\mathrm{hmf}}(x))) u|^2-|\nabla |u||^2\big) e^{\xi(\tau,x)}\,dx\\
&\quad -2\int_{\Omega} \big|\nabla |u|+\frac12 |u| \nabla \xi \big|^2 e^{\xi(\tau,x)}\,dx.
\end{split}
\end{equation*}
On the other hand, the diamagnetic inequality shows that
  \begin{equation*}
  \begin{split}
|\nabla |u|(x)|=\Big|\Re\big(\frac{\bar{u}(x)}{|u(x)|}\nabla u(x)\big) \Big|&= \Big|\Re\Big(\big(\nabla u(x)+i(A_B+A_{\mathrm{hmf}}) u(x)\big)\frac{\bar{u}(x)}{|u(x)|}\Big) \Big|\\
&\leq  \Big|\big(\nabla +i(A_B+A_{\mathrm{hmf}})\big) u(x)\Big|.
\end{split}
 \end{equation*}
Therefore,  we finally prove that
 \begin{equation*}
\begin{split}
\frac{d J_\Omega(\tau)}{d\tau}&\leq 0
\end{split}
\end{equation*}
Consequently, \eqref{est:J'} follows, which implies Lemma \ref{lem: J}.

\end{proof}

\subsection{Method II}

In this subsection, we will use the Schulman-Sunada formula (which is the second method we used in \cite{WZZ} to construct the Schr\"odinger propagator)
to reconstruct the heat propagator. For more about the Schulman-Sunada formula, we refer the reads to \cite{stov,stov1}.
The representation and estimate of the heat kernel capture more magnetic effects.

Let $M=\R^2\setminus\{\vec{0}\}=(0,+\infty)\times \mathbb{S}^1$ where $\mathbb{S}^1$ is the unit circle. The universal covering space of $M$ is $\tilde{M}=(0,+\infty)\times \R$, then $M=\tilde{M}/\Gamma$ where the structure group $\Gamma=2\pi\Z$ acts in the second factor of the Cartesian product. Then Schulman's ansatz (see \cite{stov,stov1}) enables us to compute the heat propagator $e^{-t H_{\alpha, B_0}}$ on $M$ by using the heat propagator $e^{-t \tilde{H}_{\alpha, B_0}}$ (see the operator $\tilde{H}_{\alpha, B_0}$ in \eqref{t-H} below) on $\tilde{M}$. More precisely, see \cite[(1)]{stov1}, we have
\begin{equation}\label{HH}
e^{-t H_{\alpha, B_0}}(r_1,\theta_1; r_2, \theta_2)=\sum_{j\in\Z} e^{-t \tilde{H}_{\alpha, B_0}} (r_1,\theta_1+2j\pi; r_2, \theta_2),
\end{equation}
which is similar to the construction of wave propagator on $\mathbb{T}^n$, see \cite[(3.5.12)]{sogge}. In the following subsections, we will use it to construct the
heat kernel.\vspace{0.1cm}

\begin{proposition}\label{prop:H'}
Let $K_H(t; x,y)$ be the heat kernel in \eqref{rep:h} of Proposition \ref{prop:H} .
Then
\begin{equation}\label{heatkernel:2}
\begin{split}
&e^{-t H_{\alpha, B_0}}(r_1,\theta_1; r_2, \theta_2)\\
&=\frac{B_0}{4\pi \sinh (tB_0)}e^{-\frac{B_0(r_1^2+r_2^2)}{4\tanh (tB_0)}} \\&\qquad\times \bigg( e^{\frac{B_0r_1r_2}{2\sinh (tB_0)}\cosh(tB_0+i(\theta_1-\theta_2))} e^{-\alpha tB_0} e^{-i\alpha(\theta_1-\theta_2)}\varphi(\theta_1,\theta_2)
\\&\quad-\frac{\sin(\alpha\pi)}{\pi }\int_{-\infty}^{\infty} e^{-\frac{B_0r_1r_2}{2\sinh (tB_0)}\cosh(s)} \frac{e^{s\alpha}}{e^{i(\theta_1-\theta_2)}e^{s+tB_0}+1}\, ds\bigg),
\end{split}
\end{equation}
where
\begin{equation}\label{varphi}
\begin{split}
\varphi(\theta_1,\theta_2)=\begin{cases}1,\qquad |\theta_1-\theta_2|\leq \pi \\
e^{-2\pi\alpha i},\quad  -2\pi<\theta_1-\theta_2\leq -\pi\\
e^{2\pi\alpha i},\quad  \pi<\theta_1-\theta_2\leq 2\pi.
\end{cases}
\end{split}
\end{equation}
Furthermore, we obtain the estimate
\begin{equation}\label{est:heatkernel:2}
\begin{split}
\Big| e^{-t H_{\alpha, B_0}}(x, y)\Big|\leq C\frac{B_0 e^{-\alpha tB_0}}{4\pi \sinh (tB_0)} e^{-\frac{B_0|x-y|^2}{4\tanh (tB_0)}}.
\end{split}
\end{equation}

\end{proposition}

\begin{proof}
Recall \eqref{rep:h} and $\alpha_k=|k+\alpha|$, we have
\begin{equation*}
K_H(t; x,y)=\frac{B_0}{4\pi^2\sinh tB_0} e^{-\frac{B_0(r_1^2+r_2^2)\cosh tB_0}{4\sinh tB_0}}\sum_{k\in\Z}e^{ik(\theta_1-\theta_2)}e^{-(k+\alpha)tB_0}I_{\alpha_k}\left(\frac{B_0r_1r_2}{2\sinh tB_0}\right).
\end{equation*}
The main obstacle is to take the summation in $k$. If $\alpha=0$,  by using the formula \cite[Eq. 9. 6. 19]{AS65}
\begin{equation*}
\sum_{k\in\Z} e^{kz}I_{|k|}(x)=e^{x\cosh(z)},
\end{equation*}
we will see
\begin{equation*}
K_H(t; x,y)=\frac{B_0}{4\pi\sinh(B_0t)}\exp\left(-\frac{B_0(r_1^2+r_2^2)}{4\tanh(B_0t)}+
\frac{B_0r_1r_2}{2\sinh(B_0t)}\cosh(B_0t+i(\theta_1-\theta_2))\right),\quad \text{if}\quad \alpha=0,
\end{equation*}
which is exactly the same as the result \eqref{equ:Meh} obtained from the Mehler formula. Heuristically, if we can replace the above summation in $k$ by the integration in $k$, then one can use the
translation invariant of integration to obtain further result.  To this end, as did in \cite[Section 4]{WZZ}, we consider the operator
\begin{equation}\label{t-H}
\tilde{H}_{\alpha, B_0}=-\partial_r^2-\frac{1}{r}\partial_r+\frac{1}{r^2}\Big(-i\partial_\theta+\alpha+\frac{B_0r^2}{2}\Big)^2,
\end{equation}
which acts on $L^2(\tilde{M}, rdr\, d\theta)$. We emphasize that the variable $\theta\in\R$ while not compact manifold $\mathbb{S}^1$. Then
we choose $e^{i(\tilde{k}-\alpha)\theta}$ as an eigenfunction of  the operator $\Big(-i\partial_\theta+\alpha+\frac{B_0r^2}{2}\Big)^2$ on $L^2_\theta(\R)$, which satisfies that
\begin{equation}\label{eq:ef}
\Big(-i\partial_\theta+\alpha+\frac{B_0r^2}{2}\Big)^2 \varphi(\theta)=\Big(\tilde{k}+\frac{B_0r^2}{2}\Big)^2 \varphi(\theta).
\end{equation}
It worths to point out that $\tilde{k}\in\R$ is a real number while $k\in\Z$. More importantly, we informally move the $\alpha$ in right hand side of \eqref{eq:ef} to the $e^{i(\tilde{k}-\alpha)\theta}$, which will simplify the eigenfunctions.
Hence, similarly as \cite{WZZ} for the Schr\"odinger kernel, we obtain the heat kernel of $e^{-t \tilde{H}_{\alpha, B_0}} $
\begin{equation*}
\begin{split}
\tilde{K}_H(t; x,y)
=&\frac{B_0}{4\pi \sinh (tB_0)}e^{-\frac{B_0(r_1^2+r_2^2)}{4\tanh (tB_0)}} \\&\times \int_{\R}
 e^{i(\tilde{k}-\alpha)(\theta_1-\theta_2-itB_0)}
I_{|\tilde{k}|}\bigg(\frac{B_0r_1r_2}{2\sinh (tB_0)}\bigg)\, d\tilde{k},
\end{split}
\end{equation*}
where $x=(r_1,\theta_1)\in \tilde{M}$ and $y=(r_2,\theta_2)\in \tilde{M}$.
A key identity is from \cite[(2.11)]{stov}, we state it in the following lemma.
\begin{lemma}\label{lem: key2} For any $z\in\C$, one has
\begin{equation}\label{id:key2}
\begin{split}
\int_{\R} e^{z\tilde{k}} I_{|\tilde{k}|}(x)\,  d\tilde{k}&= e^{x\cosh(z)}H(\pi-|\Im z|)\\
&+\frac{1}{2\pi i}\int_{-\infty}^{\infty} e^{-x\cosh(s)} \left(\frac1{z+s+i\pi}-\frac1{z+s-i\pi}\right)\, ds.
\end{split}
\end{equation}
where
\begin{equation*}H(x)=
\begin{cases}
1,\quad x>0,\\
0,\quad x\leq 0.
\end{cases}
\end{equation*} is the Heaviside step function.
\end{lemma}

\begin{proof}\footnote{We would like to appreciate to Prof. P. $\check{S}\check{t}$ov\'i$\check{c}$ek for his helpful discussion. }
To prove \eqref{id:key2}, we recall the representation of the modified Bessel function of order $\nu$
\begin{equation}\label{m-bessel}
I_{\nu}(x)=\frac1{\pi}\int_0^\pi e^{x\cos s} \cos(s\nu)\, ds-\frac{\sin(\nu \pi)}{\pi}\int_0^\infty e^{-x\cosh s-\nu s}\, ds,\quad x>0.
\end{equation}
For fixed $x>0$,  one has that $$I_{\nu}(x)\sim \frac{1}{\sqrt{2\pi\nu}}\Big(\frac{ex}{2\nu}\Big)^\nu,\quad \nu\to +\infty$$
decays very rapidly in the order $\nu$. Due to this fact, the LHS of \eqref{id:key2} is absolutely convergent, hence the dominated convergent theorem implies that the LHS of \eqref{id:key2} represents an entire function in $z$ (holomorphic everywhere on $\C$). The RHS of \eqref{id:key2} is also an entire function in $z$ as well but it is less obvious. The RHS of \eqref{id:key2} is surely holomorphic in $z$ everywhere on $\C$ except the lines $\Im z=\pm \pi$.
On these lines, there is no discontinuity for the RHS of \eqref{id:key2}. For example, we consider $\Im z=\pi$. In fact,  if set
\begin{equation*}
\begin{split}
F(z)=&e^{x\cosh(z)}H(\pi-|\Im z|)\\&
+\frac{1}{2\pi i}\int_{-\infty}^{\infty} e^{-x\cosh(s)} \left(\frac1{z+s+i\pi}-\frac1{z+s-i\pi}\right)\, ds,
\end{split}
\end{equation*}
then one can prove
$$\lim_{\Im z\to \pi-} F(z)=\lim_{\Im z\to \pi+} F(z).$$
Indeed, for $a\in\R$ and $\epsilon>0$, we need to prove
$$\lim_{\epsilon\to0} F(a+i\pi-i\epsilon)=\lim_{\epsilon\to0} F(a+i\pi+i\epsilon),$$
that is,
\begin{equation*}
\begin{split}
&\lim_{\epsilon\to0} \left(e^{x\cosh(a+i\pi-i\epsilon)}+\frac{1}{2\pi i}\int_{-\infty}^{\infty} e^{-x\cosh(s)} \left(\frac1{a+i\pi-i\epsilon+s+i\pi}-\frac1{a+i\pi-i\epsilon+s-i\pi}\right)\, ds\right)\\
&=\lim_{\epsilon\to0} \frac{1}{2\pi i}\int_{-\infty}^{\infty} e^{-x\cosh(s)} \left(\frac1{a+i\pi+i\epsilon+s+i\pi}-\frac1{a+i\pi+i\epsilon+s-i\pi}\right)\, ds.
\end{split}
\end{equation*}
By direct computation, we obtain
\begin{equation*}
\begin{split}
&\lim_{\epsilon\to0} \frac{1}{2\pi i}\Big[\int_{-\infty}^{\infty} e^{-x\cosh(s)} \left(\frac1{a+i\epsilon+s+2i\pi}-\frac1{a-i\epsilon+s+2i\pi}\right)\,ds\\
&\qquad+\int_{-\infty}^{\infty} e^{-x\cosh(s)} \left(\frac1{a+s-i\epsilon}-\frac1{a+s+i\epsilon}\right)\,ds\Big]\\
&=\lim_{\epsilon\to0} \frac{1}{2\pi i}\Big[\int_{-\infty}^{\infty} e^{-x\cosh(s)} \frac{-2i\epsilon}{(a+s+2i\pi)^2+\epsilon^2}\,ds
+\int_{-\infty}^{\infty} e^{-x\cosh(s)} \frac{2i\epsilon}{(a+s)^2+\epsilon^2}\,ds\Big]\\
&=\lim_{\epsilon\to0} \frac{1}{\pi}\Big[-\int_{-\infty}^{\infty} e^{-x\cosh(s)} \frac{\epsilon}{(a+s+2i\pi)^2+\epsilon^2}\,ds
+\int_{-\infty}^{\infty} e^{-x\cosh(s)} \frac{\epsilon}{(a+s)^2+\epsilon^2}\,ds\Big]\\
&=e^{-x\cosh(a)},
\end{split}
\end{equation*}
where we use the fact that the Poisson kernel is an approximation to the identity,
implying that, for any reasonable function $m(x)$
\begin{equation*}
m(x)=\lim_{\epsilon\to0}\frac1{\pi}\int_{-\infty}^{\infty} m(y) \frac{\epsilon}{(x-y)^2+\epsilon^2}\,ds.
\end{equation*}
Obviously, since $\cosh(a+i\pi)=-\cosh a$, we have
\begin{equation*}
\begin{split}
&e^{x\cosh(a+i\pi)}=e^{-x\cosh(a)}\\
=&\lim_{\epsilon\to0} \frac{1}{2\pi i}\Big[\int_{-\infty}^{\infty} e^{-x\cosh(s)} \left(\frac1{a+i\epsilon+s+2i\pi}-\frac1{a-i\epsilon+s+2i\pi}\right)\,ds\\
&\qquad+\int_{-\infty}^{\infty} e^{-x\cosh(s)} \left(\frac1{a+s-i\epsilon}-\frac1{a+s+i\epsilon}\right)\,ds\Big].
\end{split}
\end{equation*}
Therefore the RHS of \eqref{id:key2}, $F(z)$, is an entire function in $z$ as well. As a consequence, it suffices to verify the formula \eqref{id:key2} for purely imaginary value of $z$ only.
Let $z=ib$ and recall \eqref{m-bessel}, then
\begin{equation*}
\begin{split}
\int_{\R} e^{ib\tilde{k}} I_{|\tilde{k}|}(x)\,  d\tilde{k}&= \frac1{\pi}\int_{\R} e^{ib\tilde{k}} \int_0^\pi e^{x\cos s} \cos(s|\tilde{k}|)\, ds \,  d\tilde{k}\\
&\qquad-\int_{\R} e^{ib\tilde{k}}\frac{\sin(|\tilde{k}| \pi)}{\pi}\int_0^\infty e^{-x\cosh s-|\tilde{k}| s}\, ds \,  d\tilde{k}.
\end{split}
\end{equation*}
The first term becomes that
\begin{equation*}
\begin{split}
 \frac1{\pi}\int_{\R} e^{ib\tilde{k}} \int_0^\pi e^{x\cos s} \cos(s\tilde{k})\, ds \,  &d\tilde{k}= \frac1{2\pi}\int_{\R} e^{ib\tilde{k}} \int_{\R} H(\pi-|s|)e^{x\cos s} e^{-is\tilde{k}}\, ds \,  d\tilde{k}\\
 &=e^{x\cos b} H(\pi-|b|)=e^{x\cosh(z)} H(\pi-|\Im z|).
\end{split}
\end{equation*}
The second term gives
\begin{equation*}
\begin{split}
&-\int_{\R} e^{ib\tilde{k}}\frac{\sin(|\tilde{k}| \pi)}{\pi}\int_0^\infty e^{-x\cosh s-|\tilde{k}| s}\, ds \,  d\tilde{k}\\
&=-\frac1{2\pi i}\int_0^\infty e^{-x\cosh s}\Big(\int_{0}^\infty \Big(e^{[i(b+\pi)-s]\tilde{k}}- e^{[i(b-\pi)-s]\tilde{k}}\Big)\, d\tilde{k}\\
&\qquad+\int_{-\infty}^0 \Big(e^{[i(b-\pi)+s]\tilde{k}}- e^{[i(b+\pi)+s]\tilde{k}}\Big)\, d\tilde{k}\Big) \,  ds\\
&=\frac{1}{2\pi i}\int_{-\infty}^{\infty} e^{-x\cosh(s)} \left(\frac1{ib+s+i\pi}-\frac1{ib+s-i\pi}\right)\, ds.
\end{split}
\end{equation*}
Therefore, we have proved \eqref{id:key2}.
\end{proof}

Let $z=tB_0+i(\theta_1-\theta_2)$, by using Lemma \ref{lem: key2}, we obtain
\begin{equation*}
\begin{split}
&\tilde{K}_H(t; x,y)
=\frac{B_0}{4\pi \sinh (tB_0)}e^{-\frac{B_0(r_1^2+r_2^2)}{4\tanh (tB_0)}}  e^{-\alpha tB_0}e^{-i\alpha(\theta_1-\theta_2)}\\&\qquad \qquad
\times \bigg( e^{\frac{B_0r_1r_2}{2\sinh (tB_0)}\cosh(z)}H\big(\pi-|\theta_1-\theta_2|\big)
\\&\quad+\frac{1}{2\pi i}\int_{-\infty}^{\infty} e^{-\frac{B_0r_1r_2}{2\sinh (tB_0)}\cosh(s)} \left(\frac1{z+s+i\pi}-\frac1{z+s-i\pi}\right)\, ds\bigg).
\end{split}
\end{equation*}
By using \eqref{HH} and letting $z_j=tB_0+i(\theta_1+2\pi j-\theta_2)$, we further show
\begin{equation}\label{HH'}
\begin{split}
&e^{-t H_{\alpha, B_0}}(r_1,\theta_1; r_2, \theta_2)=\sum_{j\in\Z} e^{-t \tilde{H}_{\alpha, B_0}} (r_1,\theta_1+2j\pi; r_2, \theta_2),\\
&=\frac{B_0}{4\pi \sinh (tB_0)}e^{-\frac{B_0(r_1^2+r_2^2)}{4\tanh (tB_0)}} \sum_{j\in\Z}  e^{-\alpha z_j}\\&\times \bigg( e^{\frac{B_0r_1r_2}{2\sinh (tB_0)}\cosh(z_j)}H\big(\pi-|\theta_1+2\pi j-\theta_2|\big)
\\&\quad+\frac{1}{2\pi i}\int_{-\infty}^{\infty} e^{-\frac{B_0r_1r_2}{2\sinh (tB_0)}\cosh(s)} \left(\frac1{z_j+s+i\pi}-\frac1{z_j+s-i\pi}\right)\, ds\bigg).
\end{split}
\end{equation}
In the following, we consider the two summations.
Since $\theta_1,\theta_2\in [0, 2\pi)$, then $\theta_1-\theta_2\in(-2\pi, 2\pi)$, recall \eqref{varphi},  hence we obtain
\begin{equation*}
\begin{split}
&\sum_{j\in\Z}  e^{-\alpha z_j} e^{\frac{B_0r_1r_2}{2\sinh (tB_0)}\cosh(z_j)}H\big(\pi-|\theta_1+2\pi j-\theta_2|\big)\\
&=e^{\frac{B_0r_1r_2}{2\sinh (tB_0)}\cosh(z)} e^{-\alpha tB_0}\sum_{j\in\Z}  e^{-i\alpha(\theta_1-\theta_2+2\pi j)} H\big(\pi-|\theta_1+2\pi j-\theta_2|\big)\\
&=e^{\frac{B_0r_1r_2}{2\sinh (tB_0)}\cosh(z)} e^{-\alpha tB_0} e^{-i\alpha(\theta_1-\theta_2)}\varphi(\theta_1,\theta_2),
\end{split}
\end{equation*}

For the second summation term, we use the formula
\begin{equation*}
\sum_{j\in\Z} \frac{e^{-2\pi i\alpha j}}{\sigma+2\pi j}=\frac{i e^{i\alpha\sigma}}{e^{i\sigma}-1},\quad \alpha\in(0,1),\quad \sigma\in\C\setminus 2\pi\Z,
\end{equation*}
to obtain
\begin{equation*}
\begin{split}
&\sum_{j\in\Z}  e^{-\alpha z_j}  \left(\frac1{z_j+s+i\pi}-\frac1{z_j+s-i\pi}\right)\\
&= e^{-\alpha(tB_0+i(\theta_1-\theta_2))}\sum_{j\in\Z} \big(\frac{e^{-2\pi i\alpha j}}{i(\theta_1-\theta_2+2j\pi+\pi)+s+tB_0}-\frac{e^{-2\pi i\alpha j}}{i(\theta_1-\theta_2+2j\pi-\pi)+s+tB_0}\big)\\
&=-i e^{-\alpha(tB_0+i(\theta_1-\theta_2))}\sum_{j\in\Z} \big(\frac{e^{-2\pi i\alpha j}}{\sigma_1+2j\pi}-\frac{e^{-2\pi i\alpha j}}{\sigma_2+2j\pi}\big)\\
&= e^{-\alpha(tB_0+i(\theta_1-\theta_2))}\big(\frac{e^{i\alpha \sigma_1}}{e^{i\sigma_1}-1}-\frac{e^{i\alpha \sigma_2}}{e^{i\sigma_2}-1}\big)= \frac{e^{s\alpha}}{e^{i(\theta_1-\theta_2+\pi)}e^{s+tB_0}-1}\big(e^{i\alpha\pi}-e^{-i\alpha\pi}\big)\\
&= -2i\sin(\alpha\pi)\frac{e^{s\alpha}}{e^{i(\theta_1-\theta_2)}e^{s+tB_0}+1},
\end{split}
\end{equation*}
where $\sigma_1=(\theta_1-\theta_2+\pi) -i(s+tB_0)$ and $\sigma_2=(\theta_1-\theta_2-\pi) -i(s+tB_0)$.
Therefore, by using \eqref{HH'}, we show \eqref{heatkernel:2}
\begin{equation}\label{heat:rep}
\begin{split}
&e^{-t H_{\alpha, B_0}}(r_1,\theta_1; r_2, \theta_2)\\
&=\frac{B_0}{4\pi \sinh (tB_0)}e^{-\frac{B_0(r_1^2+r_2^2)}{4\tanh (tB_0)}} \\&\qquad\times \bigg( e^{\frac{B_0r_1r_2}{2\sinh (tB_0)}\cosh(tB_0+i(\theta_1-\theta_2))} e^{-\alpha tB_0} e^{-i\alpha(\theta_1-\theta_2)}\varphi(\theta_1,\theta_2)
\\&\quad-\frac{\sin(\alpha\pi)}{\pi }\int_{-\infty}^{\infty} e^{-\frac{B_0r_1r_2}{2\sinh (tB_0)}\cosh(s)} \frac{e^{s\alpha}}{e^{i(\theta_1-\theta_2)}e^{s+tB_0}+1}\, ds\bigg).
\end{split}
\end{equation}
To prove \eqref{est:heatkernel:2},
we first note that
\begin{equation*}
\begin{split}
\cosh(tB_0+i\theta)&=\cos\theta\cosh(tB_0)+i\sin\theta\sinh(tB_0), \\|x-y|^2&=r_1^2+r_2^2-2r_1r_2\cos(\theta_1-\theta_2),
\end{split}
\end{equation*}
the first term is controlled by
\begin{equation}\label{est:h1}
\begin{split}
&\frac{B_0}{4\pi \sinh (tB_0)}e^{-\frac{B_0(r_1^2+r_2^2)}{4\tanh (tB_0)}} e^{\frac{B_0r_1r_2}{2\sinh (tB_0)}\cosh(tB_0)\cos(\theta_1-\theta_2))} e^{-\alpha tB_0} \\
&\leq C \frac{B_0 e^{-\alpha tB_0}}{4\pi \sinh (tB_0)}e^{-\frac{B_0|x-y|^2}{4\tanh (tB_0)}}.
\end{split}
\end{equation}
For the second term, we aim to prove
\begin{align}\label{est:h2}
\Big|\int_{\R} e^{-\frac{B_0r_1r_2}{2\sinh (tB_0)}\cosh s}\frac{e^{\alpha s}}{1+e^{s+i(\theta_1-\theta_2)+tB_0}}\,ds\Big|\leq C,
\end{align}
where $C$ is a constant independent of $t$, $r_1, r_2$ and $\theta_1, \theta_2$. To this end, let $\theta=\theta_1-\theta_2$,
we write
\begin{align*}
&\int_{\R} e^{-\frac{B_0r_1r_2}{2\sinh (tB_0)}\cosh s}\frac{e^{\alpha s}}{1+e^{s+i(\theta_1-\theta_2)+tB_0}}\,ds\\
&=e^{-\alpha tB_0}\int_{0}^\infty \Big( e^{-\frac{B_0r_1r_2}{2\sinh (tB_0)}\cosh (-s-tB_0)}\frac{e^{-\alpha s}}{1+e^{-s+i\theta}}+e^{-\frac{B_0r_1r_2}{2\sinh (tB_0)}\cosh (s-tB_0)}\frac{e^{\alpha s}}{1+e^{s+i\theta}}\Big)\,ds\\
&=e^{-\alpha tB_0}\Big( \int_{0}^\infty e^{-\frac{B_0r_1r_2}{2\sinh (tB_0)}\cosh (-s-tB_0)}\big(\frac{e^{-\alpha s}}{1+e^{-s+i\theta}}+\frac{e^{\alpha s}}{1+e^{s+i\theta}}\big)\,ds\\
&\quad +\int_{0}^\infty \Big(e^{-\frac{B_0r_1r_2}{2\sinh (tB_0)}\cosh (s-tB_0)}-e^{-\frac{B_0r_1r_2}{2\sinh (tB_0)}\cosh (-s-tB_0)}\Big)\frac{e^{\alpha s}}{1+e^{s+i\theta}} \,ds\Big),
\end{align*}
then we just need to verify that
\begin{align}\label{est: term1}
\int_{0}^\infty\Big|\frac{e^{-\alpha s}}{1+e^{-s+i\theta}}+\frac{e^{\alpha s}}{1+e^{s+i\theta}}\Big| \,ds\lesssim 1,
\end{align}
and
\begin{align}\label{est: term2}
\Big|\int_{0}^\infty \Big(e^{-\frac{B_0r_1r_2}{2\sinh (tB_0)}\cosh (s-tB_0)}-e^{-\frac{B_0r_1r_2}{2\sinh (tB_0)}\cosh (-s-tB_0)}\Big)\frac{e^{\alpha s}}{1+e^{s+i\theta}} \,ds\Big|\lesssim1,
\end{align}
where the implicit constant is independent of $\theta$.
We first prove \eqref{est: term1}. In fact,
\begin{align*}
&\frac{e^{-\alpha s}}{1+e^{-s+i\theta}}+\frac{e^{\alpha s}}{1+e^{s+i\theta}}\\
&=\frac{\cosh(\alpha s)e^{-i\theta}+\cosh((1-\alpha)s)}{\cos\theta+\cosh s}\\
&=\frac{\cosh(\alpha s)\cos\theta+\cosh((1-\alpha)s)-i\sin\theta\cosh(\alpha s)}{2(\cos^2(\frac{\theta}{2})+\sinh^2(\frac s 2))}\\
&=\frac{2\cos^2(\frac{\theta}2)\cosh(\alpha s)+(\cosh((1-\alpha)s)-\cosh(\alpha s))-2i\sin(\frac\theta 2)\cos(\frac\theta 2)\cosh(\alpha s)}{2(\cos^2(\frac{\theta}{2})+\sinh^2(\frac s 2))}.
\end{align*}
Since $\cosh x-1\sim\frac{x^2}2,\sinh x\sim x$,  as $x\to 0;$ $\cosh x\sim e^x,\sinh x\sim e^{x}$, as $x\to \infty,$
 we have
\begin{align*}
\int_{0}^\infty\Big|\frac{\cos^2(\frac{\theta}2)\cosh(\alpha s)}{\cos^2(\frac{\theta}{2})+\sinh^2(\frac s 2)}\Big| \,ds\lesssim\int_0^1\frac{2|\cos(\frac\theta 2)|}{s^2+(2|\cos(\frac\theta 2)|)^2}ds+\int_1^\infty\ e^{(\alpha-1)s}ds\lesssim 1.
\end{align*}
Similarly,  we obtain
\begin{align*}
\int_{0}^\infty\Big|\frac{\sin(\frac\theta 2)\cos(\frac\theta 2)\cosh(\alpha s)}{\cos^2(\frac{\theta}{2})+\sinh^2(\frac s 2)}\Big| \,ds\lesssim 1.
\end{align*}
Finally, we verify that
\begin{align*}
&\int_{0}^\infty\Big|\frac{\cosh((1-\alpha)s)-\cosh(\alpha s)}{\cos^2(\frac{\theta}{2})+\sinh^2(\frac s 2)}\Big| \,ds\\
&\lesssim\int_0^1\frac{|\frac{(1-\alpha)^2}2-\frac{\alpha^2}2|s^2}{s^2}ds+\int_1^\infty \big( e^{-\alpha s}+e^{(\alpha-1)s} \big)ds\lesssim 1.
\end{align*}
We next prove \eqref{est: term2}.
For convenience, we denote
\begin{equation*}
f(s)=e^{-\frac{B_0r_1r_2}{2\sinh(tB_0)}.
\cosh(s-tB_0)}
\end{equation*}
Then, by noting that $|f(\pm s)|\leq1,$ we obtain
\begin{align*}
\Big|\int_{1}^\infty \big(f(s)-f(-s)\big)\frac{e^{\alpha s}}{1+e^{s+i\theta}} \,ds\Big|
&\leq\int_1^\infty \Big|\frac{e^{\alpha s}}{1+e^{s+i\theta}}\Big|ds\\
&=\int_1^\infty \Big|\frac{e^{(\alpha-1) s}}{e^{-s}+e^{i\theta}}\Big|ds\\
&\leq\frac{1}{1-e^{-1}}\int_1^\infty e^{(\alpha-1)s}ds\\
&\lesssim 1,
\end{align*}
hence, for \eqref{est: term2}, it suffices to prove
\begin{align*}
\int_0^1|f(s)-f(-s)|\Big|\frac{e^{\alpha s}}{1+e^{s+i\theta}}\Big|\,ds \lesssim1.
\end{align*}
Since
\begin{align*}
f'(\pm s)&=\mp\frac{B_0r_1r_2}{2\sinh(tB_0)}\sinh(\pm s-tB_0)f(\pm s)\\
&=\mp\frac{B_0r_1r_2}{2}\Big(\frac{\pm\sinh s\cosh(tB_0)}{\sinh(tB_0)}-\cosh s\Big)e^{-\frac{B_0r_1r_2}{2\sinh(tB_0)}\cosh(\pm s-tB_0)},
\end{align*}
then for $0<s<1,$ there exists a constant $C$ such that $|f'(s)|\leq C$. Hence, by differential mean value theorem, $|f(s)-f(-s)|\leq Cs$ for $0<s<1$, thus we have
\begin{align*}
\int_0^1|f(s)-f(-s)|\Big|\frac{e^{\alpha s}}{1+e^{s+i\theta}}\Big|\,ds \leq C\int_0^1 \frac{se^{\alpha s}}{e^{s}-1}ds\lesssim1.
\end{align*}
Therefore,  we prove \eqref{est:h2}.
By collecting \eqref{heat:rep}, \eqref{est:h1} and \eqref{est:h2}, we finally obtain
\begin{equation*}
\begin{split}
&\Big| e^{-t H_{\alpha, B_0}}(r_1,\theta_1; r_2, \theta_2)\Big|\\
&\leq C\frac{B_0 e^{-\alpha tB_0}}{4\pi \sinh (tB_0)} \left(e^{-\frac{B_0|x-y|^2}{4\tanh (tB_0)}}+e^{-\frac{B_0(r_1^2+r_2^2)}{4\tanh (tB_0)}} \right).
\end{split}
\end{equation*}
which implies \eqref{est:heatkernel:2}.

\end{proof}

\section{Bernstein inequalities and square function inequalities }\label{sec:BS}
In this section, we prove the Bernstein inequalities and the square function inequality associated with the Schr\"odinger operator $H_{\alpha,B_0}$
by using the heat kernel estimates showed in the previous section.

\begin{proposition}[Bernstein inequalities]
Let $\varphi(\lambda)$ be a $C^\infty_c$ bump function on $\R$  with support in $[\frac{1}{2},2]$, then it holds for any $f\in L^q(\mathbb{R}^2)$ and $j\in\mathbb{Z}$
\begin{equation}\label{est:Bern}
\|\varphi(2^{-j}\sqrt{H_{\alpha,B_0}})f\|_{L^p(\mathbb{R}^2)}\lesssim2^{2j\big(\frac{1}{q}-\frac{1}{p}\big)}\|f\|_{L^q(\mathbb{R}^2)},\quad 1\leq q\leq p\leq\infty.
\end{equation}
\end{proposition}

\begin{proof}
Let  $\psi(x)=\varphi(\sqrt{x})$ and $\psi_e(x):=\psi(x)e^{2x}$. Then $\psi_e$  is  a $C^\infty_c$-function on $\R$ with support in $[\frac{1}{4},4]$ and
then its Fourier
transform $\hat{\psi}_e$ belongs to Schwartz class. We write
\begin{align*}
\varphi(\sqrt{x})&=\psi(x)= e^{-2x}\psi_{e}(x)\\
&=e^{-2x}\int_{\R} e^{i x \cdot\xi} \hat{\psi}_e(\xi)\, d\xi\\
 &=e^{-x}\int_{\R} e^{-x(1-i\xi)} \hat{\psi}_e(\xi)\, d\xi.
\end{align*}
Therefore, by the functional calculus, we obtain
\begin{align*}
 &\varphi(\sqrt{H_{\alpha,B_0}})=\psi(H_{\alpha,B_0})=e^{-H_{\alpha,B_0}}\int_{\R} e^{-(1-i\xi)H_{\alpha,B_0}} \hat{\psi}_e(\xi)\, d\xi,
\end{align*}
furthermore,
\begin{align*}
 &\varphi(2^{-j}\sqrt{H_{\alpha,B_0}})=\psi(2^{-2j}H_{\alpha,B_0})=e^{-2^{-2j}H_{\alpha,B_0}}\int_{\R} e^{-(1-i\xi)2^{-2j}H_{\alpha,B_0}} \hat{\psi}_e(\xi)\, d\xi.
\end{align*}
By using \eqref{est:G-H} in Proposition \ref{prop:Gassian-H} with $t=2^{-2j}$, we have
\begin{align*}
\Big|\varphi(2^{-j}\sqrt{H_{\alpha,B_0}})(x,y)\Big|&\lesssim 2^{4j}\int_{\R^2}e^{-\frac{2^{2j}|x-z|^2}{C}}e^{-\frac{2^{2j}|y-z|^2}{C}}\, dz \int_{\R} \hat{\psi}_e(\xi)\, d\xi\\
& \lesssim 2^{2j}\int_{\R^2}e^{-\frac{|2^jx-z|^2}{C}}e^{-\frac{|2^jy-z|^2}{C}}\, dz\\
& \lesssim 2^{2j}(1+2^j|x-y|)^{-N},\quad \forall N\geq 1.
\end{align*}
where we use the fact that $|\alpha-z|^2+|\beta-z|^2\geq \frac12|\alpha-\beta|^2$ with $\alpha,\beta\in\R^2$ and
\begin{align*}
\int_{\R^2}e^{-\frac{|\alpha-z|^2}{C}}e^{-\frac{|\beta-z|^2}{C}}\, dz &\lesssim e^{-\frac{|\alpha-\beta|^2}{4C}}\int_{\R^2}e^{-\frac{|\alpha-z|^2}{2C}}e^{-\frac{|\beta-z|^2}{2C}}\, dz\\
&\lesssim e^{-\frac{|\alpha-\beta|^2}{4C}}\leq (1+|\alpha-\beta|)^{-N}, \forall N\geq 1.
\end{align*}
By Young's inequality,  it follows \eqref{est:Bern}.
\end{proof}

\begin{proposition}[The square function inequality]\label{prop:squarefun} Let $\{\varphi_j\}_{j\in\mathbb Z}$ be a Littlewood-Paley sequence given by\eqref{LP-dp}.
Then for $1<p<\infty$,
there exist constants $c_p$ and $C_p$ depending on $p$ such that
\begin{equation}\label{square}
c_p\|f\|_{L^p(\R^2)}\leq
\Big\|\Big(\sum_{j\in\Z}|\varphi_j(\sqrt{H_{{\alpha}, B_0}})f|^2\Big)^{\frac12}\Big\|_{L^p(\R^2)}\leq
C_p\|f\|_{L^p(\R^2)}.
\end{equation}

\end{proposition}

\begin{proof}
By using \eqref{est:G-H} in Proposition \ref{prop:Gassian-H}, Proposition \ref{prop:squarefun} follows from
the Rademacher functions argument in \cite{Stein}. We also refer the reader to \cite{Alex} for result that the square function inequality \eqref{square} can be derived from
the heat kernel with Gaussian upper bounds.
\end{proof}

\section{The decay estimates }\label{sec:decay}

In this section, we mainly prove the decay estimate \eqref{dis-w}.
The first key ingredient is the following Proposition about the subordination formula from \cite{MS, DPR}.
\begin{proposition}
If $\varphi(\lambda)\in C_c^\infty(\mathbb{R})$ is supported in $[\frac{1}{2},2]$, then, for all $j\in\Z, t, x>0$ with $2^jt\geq 1$,  we can write
\begin{equation}\label{key}
\begin{split}
&\varphi(2^{-j}\sqrt{x})e^{it\sqrt{x}}\\&=\rho\big(\frac{tx}{2^j}, 2^jt\big)
+\varphi(2^{-j}\sqrt{x})\big(2^jt\big)^{\frac12}\int_0^\infty \chi(s,2^jt)e^{\frac{i2^jt}{4s}}e^{i2^{-j}tsx}\,ds,
\end{split}
\end{equation}
where $\rho(s,\tau)\in\mathcal{S}(\mathbb{R}\times\mathbb{R})$ is a Schwartz function and
and  $\chi\in C^\infty(\mathbb{R}\times\mathbb{R})$ with $\text{supp}\,\chi(\cdot,\tau)\subseteq[\frac{1}{16},4]$ such that
\begin{equation}\label{est:chi}
\sup_{\tau\in\R}\big|\partial_s^\alpha\partial_\tau^\beta \chi(s,\tau)\big|\lesssim_{\alpha,\beta}(1+|s|)^{-\alpha},\quad \forall \alpha,\beta\geq0.
\end{equation}
\end{proposition}

If this has been done, then by the spectral theory for the non-negative self-adjoint operator $H_{\alpha,B_0}$, we can have the representation of the micolocalized
half-wave propagator
\begin{equation}\label{key-operator}
\begin{split}
&\varphi(2^{-j}\sqrt{H_{\alpha,B_0}})e^{it\sqrt{H_{\alpha,B_0}}}\\&=\rho\big(\frac{tH_{\alpha,B_0}}{2^j}, 2^jt\big)
+\varphi(2^{-j}\sqrt{H_{\alpha,B_0}})\big(2^jt\big)^{\frac12}\int_0^\infty \chi(s,2^jt)e^{\frac{i2^jt}{4s}}e^{i2^{-j}tsH_{\alpha,B_0}}\,ds.
\end{split}
\end{equation}

\begin{proof} This result is original from \cite{MS} and the authors of \cite{DPR} provided an independent proof of the key formula.
For the self-contained and the convenience of the reader, we follow the idea of \cite{DPR} to provide the details of the proof.\vspace{0.1cm}

The starting point of the proof is the subordination formula
\begin{equation}\label{change:subord}
e^{-y\sqrt{x}}=\frac{y}{2\sqrt{\pi}}\int_0^\infty e^{-sx-\frac{y^2}{4s}}s^{-\frac{3}{2}}ds,\quad x,y>0.
\end{equation}
Indeed,  let $|\xi|=\sqrt{x}$, we have
\begin{equation*}
\begin{split}
e^{-y\sqrt{x}}&=e^{-y|\xi|}=\int_{\R} e^{-2\pi i t\cdot\xi} \int_{\R} e^{2\pi i t\cdot\eta} e^{-y|\eta|}\,d\eta \, dt\\
&=\int_{\R} e^{-2\pi i t\cdot\xi} \Big(\int_{-\infty}^0 e^{2\pi i t\cdot\eta} e^{y\eta}\, d\eta+\int_{0}^\infty e^{2\pi i t\cdot\eta} e^{-y\eta}\,d\eta\Big) \, dt\\
&=2\int_{\R} e^{-2\pi i t\cdot\xi}\frac{y}{y^2+(2\pi t)^2} \, dt=2\int_{\R} e^{-2\pi i yt\cdot\xi}\frac{1}{1+(2\pi t)^2} \, dt\\
&=2\int_{\R} e^{-2\pi i yt\cdot\xi}\int_0^\infty e^{-r(1+(2\pi t)^2)}\, dr \, dt\\
&=2\int_0^\infty e^{-r}\int_{\R} e^{-2\pi i yt\cdot\xi}e^{-r(2\pi t)^2}\, dt\, dr \\
&=2\int_0^\infty \frac{e^{-r} e^{-\frac{\xi^2y^2}{4 r}}}{\sqrt{4\pi r}}\, dr=2\int_0^\infty \frac{e^{-r} e^{-\frac{ xy^2}{4 r}}}{\sqrt{4\pi r}}\, dr,\quad s=\frac{y^2}{4 r}, \xi=\sqrt{x}, \\
&=\frac y{2\sqrt{\pi}}\int_0^\infty e^{-\frac{y^2}{4 s}} e^{-xs} s^{-\frac32}\, ds \\
\end{split}
\end{equation*}
where we use
\begin{equation*}
\alpha^{-\frac n2}e^{-\pi|\xi|^2/\alpha}=\int_{\mathbb{R}^n}e^{-2\pi ix\xi}e^{-\pi \alpha |x|^2}dx.
\end{equation*}
To obtain $e^{it\sqrt{x}}$, we extend \eqref{change:subord} by setting $y=\epsilon-it$ with $\epsilon>0$
\begin{equation}\label{identity}
\begin{split}
e^{it\sqrt{x}}&=\lim_{\epsilon\to 0^+}e^{-(\epsilon-it)\sqrt{x}}\\
&=\lim_{\epsilon\to 0}\frac{\epsilon-it}{2\sqrt{\pi}}\int_0^\infty e^{-xs}e^{\frac{(it-\epsilon)^2}{4s}}s^{-\frac{3}{2}}ds,\quad s=r(\epsilon-it)\\
&=\lim_{\epsilon\to 0}\frac{\sqrt{\epsilon-it}}{2\sqrt{\pi}}\int_0^\infty e^{rx(it-\epsilon)}e^{\frac{it-\epsilon}{4r}}r^{-\frac{3}{2}}dr\\
&=\lim_{\epsilon\to 0}\frac{\sqrt{\epsilon-it}}{2\sqrt{\pi}}\int_0^\infty e^{itrx}e^{-\epsilon rx}e^{\frac{it}{4r}}e^{-\frac{\epsilon}{4r}} r^{-\frac{3}{2}}dr\\
&=\lim_{\epsilon\to 0}\frac{\sqrt{\epsilon-it}}{2\sqrt{\pi}}I_{\epsilon, \epsilon x}(tx,t),
\end{split}
\end{equation}
where
\begin{equation*}
I_{\epsilon, \delta}(a,t):=\int_0^\infty e^{ira}e^{-\delta r}e^{\frac{it}{4r}}e^{-\frac{\epsilon}{4r}} r^{-\frac{3}{2}}dr.
\end{equation*}
By the dominate convergence theorem, we have that
\begin{equation*}
\begin{split}
e^{it\sqrt{x}}
&=\lim_{\epsilon\to 0}\frac{\sqrt{\epsilon-it}}{2\sqrt{\pi}}I_{\epsilon, \epsilon x}(tx,t)=\sqrt{\frac{t}{4\pi}}e^{-\frac\pi 4i}\lim_{\epsilon\to0}I_{\epsilon, \epsilon x}(tx,t).
\end{split}
\end{equation*}
Thus it suffices to consider the oscillation integral
\begin{equation}\label{limit}
\begin{split}
\lim_{\epsilon\to 0}I_{\epsilon, \epsilon x}(a,t)=I_{0, 0}(a,t)=\int_0^\infty e^{ira} e^{\frac{it}{4r}} r^{-\frac{3}{2}}dr.
\end{split}
\end{equation}

\begin{lemma}\label{lem: stationary} Let
\begin{equation*}
\begin{split}
I(a,t)=\int_0^\infty e^{ira} e^{\frac{it}{4r}} r^{-\frac{3}{2}}dr.
\end{split}
\end{equation*}
Then we can write
\begin{equation}\label{est: stat}
I(a,t)=\tilde{\rho}(a, t)+\int_0^\infty e^{ira} e^{\frac{it}{4r}} \tilde{\chi} (r)\, dr,
\end{equation}
where $\tilde{\chi}(r)\in C_0^\infty(r)$ and $\mathrm{supp}\,\tilde{\chi} \subset [\frac1{16},4]$ and $\tilde{\rho}(a, t)$ satisfies
\begin{equation}\label{est:rho}
\big|\partial_a^\alpha\partial_t^\beta\tilde{\rho}(a, t)\big|\leq C_{N,\alpha,\beta} (a+t)^{-N}, \quad \frac14\leq \frac at\leq 4, \,  t\geq 1, \forall N\geq 1.
\end{equation}
\end{lemma}
We now assume this lemma to prove \eqref{key}. By \eqref{identity} and \eqref{limit} and noticing
$$I(a,t)=2^{\frac j2}I(2^{-j}a, 2^jt),$$ we have
that \begin{equation*}
\varphi(2^{-j}\sqrt{x})e^{it\sqrt{x}}=\sqrt{\frac{t}{4\pi}}e^{-\frac\pi 4i}\varphi(2^{-j}\sqrt{x}) 2^{\frac j2} I\big(\frac{tx}{2^j}, 2^jt\big).
\end{equation*}
By the support of $\varphi$, one has $2^{2j-2}\leq x\leq 2^{2j+2}$ , hence $\frac14\leq \frac{tx}{2^j}/2^jt=x/2^{2j}\leq 4$. Note the condition $2^jt\geq 1$.
Therefore, by using this lemma, we prove
\begin{equation*}
\begin{split}
&\varphi(2^{-j}\sqrt{x})e^{it\sqrt{x}}\\=&\frac1{\sqrt{4\pi}}e^{-\frac\pi 4i}\big(2^jt\big)^{\frac12}\varphi(2^{-j}\sqrt{x}) \Big(\tilde{\rho}\big(\frac{tx}{2^j}, 2^jt\big)
+\int_0^\infty \tilde{\chi}(s)e^{\frac{i2^jt}{4s}}e^{i2^{-j}txs}\,ds\Big).
\end{split}
\end{equation*}
We need consider this expression when $2^jt\geq 1$. To this end, let $\phi\in C^\infty([0,+\infty)$ satisfies $\phi(t)=1$ if $t\ge1$ and $\phi(t)=0$ if $0\leq t\leq \frac12$,
then set $$\rho\big(\frac{tx}{2^j}, 2^jt\big)=e^{-\frac\pi 4i}\big(2^jt\big)^{\frac12}\varphi(2^{-j}\sqrt{x})\tilde{\rho}\big(\frac{tx}{2^j}, 2^jt\big)\phi(2^jt). $$
This together with \eqref{est:rho} shows
\begin{equation*}
\big|\partial_a^\alpha\partial_t^\beta{\rho}(a, t)\big|\leq C_{N,\alpha,\beta}(1+ (a+t))^{-N}, \quad  \forall N\geq 0.
\end{equation*}
which implies
$\rho(a,t)\in \mathcal{S}(\R_+\times\R_+)$.
Set $$\chi\big(s, 2^jt\big)=e^{-\frac\pi 4i}\tilde{\chi}\big(s\big)\phi(2^jt),$$
 then $\chi$ satisfies \eqref{est:chi}.
 then we finally write
\begin{equation*}
\begin{split}
&\varphi(2^{-j}\sqrt{x})e^{it\sqrt{x}}\\=&\rho\big(\frac{tx}{2^j}, 2^jt\big)
+\big(2^jt\big)^{\frac12}\varphi(2^{-j}\sqrt{x}) \int_0^\infty \chi(s,2^jt)e^{\frac{i2^jt}{4s}}e^{i2^{-j}txs}\,ds,
\end{split}
\end{equation*}
which proves \eqref{key} as desired.
\end{proof}

\begin{proof}[The proof of Lemma \ref{lem: stationary}] To prove \eqref{est: stat}, we divide the integral into three pieces.
Let $\beta(r)$ be a function in $C^\infty(\R)$ compact supported in $[\frac1{2}, 2]$ such that
$$1=\sum_{j\in\Z}\beta_{j}(r), \quad \beta_j(r)=\beta(2^{-j}r).$$
Corresponding to $\beta_j$, we decompose
$$I(a,t)=\sum_{j\in\Z} I_j(a,t)=I_{l}(a,t)+I_{m}(a,t)+I_{h}(a,t)$$ where
\begin{equation*}
\begin{split}
I_{l}(a,t)=\sum_{j\leq -5} I_j(a,t),\quad &I_{m}(a,t)=\sum_{-4\le j\leq 1} I_j(a,t),\quad I_{h}(a,t)=\sum_{j\geq 2} I_j(a,t),\\
&I_j(a,t)=\int_0^\infty e^{ira} e^{\frac{it}{4r}}\beta_j(r) r^{-\frac{3}{2}}dr.
\end{split}
\end{equation*}
Define the phase function $\phi_{a, t}(r)=ra+\frac t{4r}$, then
$$\phi'_{a,t}(r)=a-\frac{t}{4r^2}.$$
Define
$$\tilde{\rho}(a,t)=I_l(a, t)+I_h(a, t),$$
we aim to prove \eqref{est:rho}.
We first consider $I_h(a, t)$.
\begin{equation*}
\begin{split}
\big|\partial_a^\alpha\partial_t^\beta I_h(a, t)\big|&=\sum_{j\geq 2}\Big|\int_0^\infty e^{ira} e^{\frac{it}{4r}}\beta_j(r) r^{-\frac{3}{2}+\alpha}\big(\frac 1{4r}\big)^{\beta} dr\Big|\\
&\leq C \sum_{j\geq 2}\int_0^\infty  \Big|\frac{d}{dr}\Big[\Big(\frac1{\phi'_{a,t}(r)}\frac{d}{dr}\Big)^{N-1}\Big(\frac1{\phi'_{a,t}(r)}\beta_j(r) r^{-\frac{3}{2}+\alpha-\beta}\Big)\Big]\Big|\, dr.
\end{split}
\end{equation*}
Due to that $\mathrm{supp}\,\beta_j\subset [2^{j-1}, 2^{j+1}]$ with $j\geq 2$ implies $r\in [2, \infty)$ and the assumption $\frac{a}t\geq \frac14$,  we see that
$$\big|\phi'_{a,t}(r)\big|=\big|a-\frac{t}{4r^2}\big|\geq \frac{a+t}{16}.$$
For choosing $N>\alpha$, we notice the fact that
$$\Big|\Big(\frac{d}{dr}\Big)^{K}\Big(\frac1{\phi'_{a,t}(r)}\Big)\Big|\leq C_K (a+t)^{-1} r^{-K}, \quad t, r\in [1,+\infty)$$
to obtain
$$  \Big|\frac{d}{dr}\Big[\Big(\frac1{\phi'_{a,t}(r)}\frac{d}{dr}\Big)^{N-1}\Big(\frac1{\phi'_{a,t}(r)}\beta_j(r) r^{-\frac{3}{2}+\alpha-\beta}\Big)\Big]\Big|\leq C_N (a+t)^{-N}2^{-\frac 32j}.$$
Therefore, we prove
\begin{equation*}
\big|\partial_a^\alpha\partial_t^\beta I_h(a, t)\big|
\leq C_N  (a+t)^{-N} \sum_{j\geq 2}  2^{-\frac12j}\,\leq C_N  (a+t)^{-N},\quad N\geq \alpha.
\end{equation*}
Next we consider $I_l(a, t)$.
\begin{equation*}
\big|\partial_a^\alpha\partial_t^\beta I_l(a, t)\big|=\sum_{j\leq -5}\Big|\int_0^\infty e^{ira} e^{\frac{it}{4r}}\beta_j(r) r^{-\frac{3}{2}+\alpha}\big(\frac 1{4r}\big)^{\beta} dr\Big|.
\end{equation*}
\begin{equation*}
\begin{split}
&\sum_{j\leq -5}\Big|\int_0^\infty e^{\frac{it}{4r}}e^{ira}{\beta}_j(r) r^{-\frac{3}{2}+\alpha-\beta} dr\Big|\\
&\leq C_N \sum_{j\leq -5}\int_0^\infty  \Big|\frac{d}{dr}\Big[\Big(\frac1{\phi'_{a,t}(r)}\frac{d}{dr}\Big)^{N-1}\Big(\frac1{\phi'_{a,t}(r)}{\beta}_j(r) r^{-\frac{3}{2}+\alpha-\beta}\Big)\Big]\Big|\, dr.
\end{split}
\end{equation*}
Due to that $\mathrm{supp}\,\beta_j\subset [2^{j-1}, 2^{j+1}]$ with $j\leq -5$ implies $r\in (0, \frac1{16}]$ and the assumption $\frac{a}t\leq 4$, we see that
$$\big|\phi'_{a,t}(r)\big|=\big|a-\frac{t}{4r^2}\big|=\frac{|4r^2a-t|}{4r^2}\geq \frac{a+t}{32r^2}.$$
For choosing $N>\alpha$, we notice the fact that
$$\Big|\Big(\frac{d}{dr}\Big)^{K}\Big(\frac1{\phi'_{a,t}(r)}\Big)\Big|\leq C_K (a+t)^{-1} r^{2-K}, \quad t\in [1,+\infty), r\in (0, \frac1{16}]$$
to obtain
\begin{equation*}
\begin{split}
& \Big|\frac{d}{dr}\Big[\Big(\frac1{\phi'_{a,t}(r)}\frac{d}{dr}\Big)^{N-1}\Big(\frac1{\phi'_{a,t}(r)}{\beta}_j(r) r^{-\frac{3}{2}+\alpha-\beta}\Big)\Big]\Big|\\
 &\leq C_N (a+t)^{-N}2^{j\big(-\frac{3}{2}+\alpha-\beta+N\big)}.
 \end{split}
\end{equation*}
Therefore, for large enough $N$ such that $-\frac{1}{2}+\alpha-\beta+N>0$,  we prove
\begin{equation*}
\begin{split}
\big|\partial_a^\alpha\partial_t^\beta I_h(a, t)\big|
&\leq C_N  (a+t)^{-N}\sum_{j\leq -5}  2^{j\big(-\frac{1}{2}+\alpha-\beta+N\big)}\\
&\leq C_N  (a+t)^{-N},
\end{split}
\end{equation*}
where we use the assumption that $\frac a t\leq 4$ and $t\geq 1$. In sum, we prove \eqref{est:rho}. Let
$$\tilde{\chi}(r)=\sum_{j=-4}^{1}\beta_j(r) r^{-\frac32},$$
then $\tilde{\chi}(r)\in C_0^\infty(r)$ and $\mathrm{supp}\,\tilde{\chi} \subset [\frac1{16},4]$.
Hence we have
\begin{equation*}
I_m(a,t)=\int_0^\infty e^{ira} e^{\frac{it}{4r}}\tilde{\chi}(r) dr.
\end{equation*}
Therefore, we complete the proof of Lemma \ref{lem: stationary}.
\end{proof}

\subsection{Decay estimates for the microlocalized half-wave propagator}
In this subsection, we mainly prove the following results
\begin{proposition} Let $2^{-j}|t|\leq \frac{\pi}{8B_0}$ and $\varphi$ be in \eqref{LP-dp}, then
\begin{equation}\label{est: mic-decay1}
\begin{split}
\big\|\varphi(2^{-j}\sqrt{H_{\alpha,B_0}})&e^{it\sqrt{H_{\alpha,B_0}}}f\big\|_{L^\infty(\R^2)}\\&\lesssim 2^{2j}\big(1+2^j|t|\big)^{-\frac12}\|\varphi(2^{-j}\sqrt{H_{\alpha,B_0}}) f\|_{L^1(\R^2)}.
\end{split}
\end{equation}
In particular, for $0<t<T$ with any finite $T$, there exists a constant $C_T$ depending on $T$ such that
\begin{equation}\label{est: mic-decay2}
\begin{split}
\big\|\varphi(2^{-j}\sqrt{H_{\alpha,B_0}})&e^{it\sqrt{H_{\alpha,B_0}}}f\big\|_{L^\infty(\R^2)}\\&\leq C_T 2^{2j}\big(1+2^j|t|\big)^{-\frac12}\|\varphi(2^{-j}\sqrt{H_{\alpha,B_0}}) f\|_{L^1(\R^2)}.
\end{split}
\end{equation}
\end{proposition}

\begin{remark} The finite $T$ can be choosen beyond  $\frac{\pi}{B_0}$. If we could prove \eqref{est: mic-decay2},
then \eqref{dis-w} follows
\begin{equation*}
\begin{split}
&\big\|e^{it\sqrt{H_{\alpha,B_0}}}f\big\|_{L^\infty(\R^2)}\leq \sum_{j\in\Z}\big\|\varphi(2^{-j}\sqrt{H_{\alpha,B_0}})e^{it\sqrt{H_{\alpha,B_0}}}f\big\|_{L^\infty(\R^2)}\\
&\leq C_T |t|^{-\frac12}\sum_{j\in\Z}2^{\frac32j}\|\varphi(2^{-j}\sqrt{H_{\alpha,B_0}}) f\|_{L^1(\R^2)}\leq C_T |t|^{-\frac12}\|f\|_{\dot{\mathcal{B}}^{3/2}_{1,1}(\R^2)}.
\end{split}
\end{equation*}
\end{remark}

We estimate  the microlocalized half-wave propagator $$ \big\|\varphi(2^{-j}\sqrt{H_{\alpha,B_0}})e^{it\sqrt{H_{\alpha,B_0}}}f\big\|_{L^\infty(\R^2)}$$
by considering two cases that: $|t|2^j\geq 1$ and $|t|2^{j}\lesssim 1$. In the following argument, we can choose $\tilde{\varphi}\in C_c^\infty((0,+\infty))$ such that $\tilde{\varphi}(\lambda)=1$ if $\lambda\in\mathrm{supp}\,\varphi$
and $\tilde{\varphi}\varphi=\varphi$. Since $\tilde{\varphi}$ has the same property of $\varphi$, without confusion, we drop off the tilde above $\varphi$ for brief. Without loss of generality, in the following argument, we assume $t>0$.
\vspace{0.2cm}

{\bf Case 1: $t2^{j}\lesssim 1$.} We remark that we consider $t2^{j}\lesssim 1$ while not $t2^{j}\leq 1$, this will be used to extend the time interval. By the spectral theorem, one has
$$\|e^{it\sqrt{H_{\alpha,B_0}}}\|_{L^2(\R^2)\to L^2(\R^2)}\leq C.$$
Indeed, by the functional calculus, for $f\in L^2$,
we can write
\begin{equation*}
\begin{split}
e^{it\sqrt{H_{\alpha,B_0}}} f&=\sum_{k\in\Z, \atop m\in\mathbb{N}} e^{it\sqrt{\lambda_{k,m}}} c_{k,m}\tilde{V}_{k,m}(x).
\end{split}
\end{equation*}
where
\begin{equation*}
\begin{split}
c_{k,m}=\int_{\mathbb{R}^2}f(y)\overline{\tilde{V}_{k,m}(y)} dy.
\end{split}
\end{equation*}
Then
\begin{equation*}
\begin{split}
\|e^{it\sqrt{H_{\alpha,B_0}}} f\|_{L^2(\R^2)}=\Big(\sum_{k\in\Z, \atop m\in\mathbb{N}} \big| e^{it\sqrt{\lambda_{k,m}}} c_{k,m}\big|^2\Big)^{1/2}
=\Big(\sum_{k\in\Z, \atop m\in\mathbb{N}} \big| c_{k,m}\big|^2\Big)^{1/2}
=\|f\|_{L^2(\R^2)}.
\end{split}
\end{equation*}
Together with this, we use the Bernstein inequality \eqref{est:Bern}
 to prove
\begin{equation*}
\begin{split}
&\big\|\varphi(2^{-j}\sqrt{H_{\alpha,B_0}})e^{it\sqrt{H_{\alpha,B_0}}}f\big\|_{L^\infty(\R^2)}\\
&\lesssim 2^{j}\|e^{it\sqrt{H_{\alpha,B_0}}}\varphi(2^{-j}\sqrt{H_{\alpha,B_0}}) f\|_{L^2(\R^2)}\\
&\lesssim 2^{j}\|\varphi(2^{-j}\sqrt{H_{\alpha,B_0}}) f\|_{L^2(\R^2)}\lesssim 2^{2j}\|\varphi(2^{-j}\sqrt{H_{\alpha,B_0}}) f\|_{L^1(\R^2)}.
\end{split}
\end{equation*}
In this case $0<t\lesssim 2^{-j}$, we have
\begin{equation}\label{<1}
\begin{split}
&\big\|\varphi(2^{-j}\sqrt{H_{\alpha,B_0}})e^{it\sqrt{H_{\alpha,B_0}}}f\big\|_{L^\infty(\R^2)}\\
&\quad\lesssim 2^{2j}(1+2^jt)^{-N}\|\varphi(2^{-j}\sqrt{H_{\alpha,B_0}}) f\|_{L^1(\R^2)},\quad \forall N\geq 0.
\end{split}
\end{equation}\vspace{0.2cm}

{\bf Case 2: $t2^{j}\geq 1$.} In this case, we can use
 \eqref{key-operator} to obtain the micolocalized
half-wave propagator
\begin{equation*}
\begin{split}
&\varphi(2^{-j}\sqrt{H_{\alpha,B_0}})e^{it\sqrt{H_{\alpha,B_0}}}\\&=\rho\big(\frac{tH_{\alpha,B_0}}{2^j}, 2^jt\big)
+\varphi(2^{-j}\sqrt{H_{\alpha,B_0}})\big(2^jt\big)^{\frac12}\int_0^\infty \chi(s,2^jt)e^{\frac{i2^jt}{4s}}e^{i2^{-j}tsH_{\alpha,B_0}}\,ds.
\end{split}
\end{equation*}
We first use the spectral theorems and the Bernstein inequality again to estimate
\begin{equation*}
\begin{split}
&\big\|\varphi(2^{-j}\sqrt{H_{\alpha,B_0}})\rho\big(\frac{tH_{\alpha,B_0}}{2^j}, 2^jt\big) f\big\|_{L^\infty(\R^2)}.
\end{split}
\end{equation*}
Indeed, since $\rho\in \mathcal{S}(\R\times\R)$, then $$\big|\rho\big(\frac{t\lambda_{k,m}}{2^j}, 2^jt\big)\big|\leq C(1+2^jt)^{-N},\quad \forall N\geq 0.$$
Therefore, we use  the Bernstein inequality and  the spectral theorems to show
\begin{equation*}
\begin{split}
&\big\|\varphi(2^{-j}\sqrt{H_{\alpha,B_0}})\rho\big(\frac{tH_{\alpha,B_0}}{2^j}, 2^jt\big) f\big\|_{L^\infty(\R^2)}\\
&\lesssim 2^{j}\Big\|\rho\big(\frac{tH_{\alpha,B_0}}{2^j}, 2^jt\big)\varphi(2^{-j}\sqrt{H_{\alpha,B_0}}) f\Big\|_{L^2(\R^2)}\\
&\lesssim 2^{j}(1+2^jt)^{-N}\Big\|\varphi(2^{-j}\sqrt{H_{\alpha,B_0}}) f\Big\|_{L^2(\R^2)}\\
&\lesssim 2^{2j}(1+2^jt)^{-N}\Big\|\varphi(2^{-j}\sqrt{H_{\alpha,B_0}}) f\Big\|_{L^1(\R^2)}.
\end{split}
\end{equation*}

Next we use the dispersive estimates of Schr\"odinger propagator (see \cite[Theorem 1.1]{WZZ})
\begin{equation*}
\begin{split}
&\big\|e^{itH_{\alpha,B_0}} f\big\|_{L^\infty(\R^2)}\leq C |\sin(tB_0)|^{-1}\big\| f\big\|_{L^1\R^2)},\quad t\neq \frac{k\pi}{B_0}, \,k\in \Z,
\end{split}
\end{equation*}
 to estimate
\begin{equation*}
\begin{split}
&\big\|\varphi(2^{-j}\sqrt{H_{\alpha,B_0}})\big(2^jt\big)^{\frac12}\int_0^\infty \chi(s,2^jt)e^{\frac{i2^jt}{4s}}e^{i2^{-j}tsH_{\alpha,B_0}}f\,ds \big\|_{L^\infty(\R^2)}.
\end{split}
\end{equation*}
For $0<t<T_0<\frac{\pi}{2B_0}$, then $\sin(tB_0)\sim tB_0$
\begin{equation*}
\begin{split}
&\big\|\varphi(2^{-j}\sqrt{H_{\alpha,B_0}})\big(2^jt\big)^{\frac12}\int_0^\infty \chi(s,2^jt)e^{\frac{i2^jt}{4s}}e^{i2^{-j}tsH_{\alpha,B_0}}f\,ds \big\|_{L^\infty(\R^2)}\\
&\lesssim\big(2^jt\big)^{\frac12}\int_0^\infty \chi(s,2^jt)|\sin(2^{-j}ts B_0)|^{-1}\,ds \big\|\varphi(2^{-j}\sqrt{H_{\alpha,B_0}})f\big\|_{L^1(\R^2)}.
\end{split}
\end{equation*}
Since $s\in [\frac1{16}, 4]$ (the compact support of $\chi$ in $s$) and $B_0>0$,  if $2^{-j}t\leq \frac{\pi}{8B_0}$, then
\begin{equation}\label{>1}
\begin{split}
&\big\|\varphi(2^{-j}\sqrt{H_{\alpha,B_0}})\big(2^jt\big)^{\frac12}\int_0^\infty \chi(s,2^jt)e^{\frac{i2^jt}{4s}}e^{i2^{-j}tsH_{\alpha,B_0}}f\,ds \big\|_{L^\infty(\R^2)}\\
&\lesssim\big(2^jt\big)^{\frac12}(2^{-j}t)^{-1}\int_0^\infty \chi(s,2^jt)\,ds \big\|\varphi(2^{-j}\sqrt{H_{\alpha,B_0}})f\big\|_{L^1(\R^2)}\\
&\lesssim 2^{2j}\big(2^jt\big)^{-\frac12}\big\|\varphi(2^{-j}\sqrt{H_{\alpha,B_0}})f\big\|_{L^1(\R^2)}\\
&\lesssim 2^{2j}\big(1+2^jt\big)^{-\frac12}\big\|\varphi(2^{-j}\sqrt{H_{\alpha,B_0}})f\big\|_{L^1(\R^2)}.
\end{split}
\end{equation}
Collecting \eqref{<1} and \eqref{>1}, it gives \eqref{est: mic-decay1}. To prove \eqref{est: mic-decay2}, we consider $0<t<T$.
For any $T>0$, there exists $j_0$ such that $2^{-j_0}T\leq \frac{\pi}{8B_0}$ with $j_0\in\Z_+$. For $j\leq j_0$, then
$2^j t\lesssim 1$, then one has \eqref{est: mic-decay2} from the first case. While for $j\geq j_0$,  if  $2^j t\lesssim 1$, one still has \eqref{est: mic-decay2} from the first case.
Otherwise, i.e. $2^j t\geq 1$, one has \eqref{est: mic-decay2} from the second case, since we always have $2^{-j}t\leq \frac{\pi}{8B_0}$ for $j\geq j_0$ and $0<t\leq T$.

\section{Strichartz estimate}\label{sec:str}
In this section, we prove the Strichartz estimates \eqref{stri:w} in Theorem \ref{thm:dispersive} by using \eqref{est: mic-decay2}.
To this end, we need a variety of the abstract Keel-Tao's Strichartz estimates
theorem (\cite{KT}).
\begin{proposition}\label{prop:semi}
Let $(X,\mathcal{M},\mu)$ be a $\sigma$-finite measured space and
$U: I=[0,T]\rightarrow B(L^2(X,\mathcal{M},\mu))$ be a weakly
measurable map satisfying, for some constants $C$ may depending on $T$, $\alpha\geq0$,
$\sigma, h>0$,
\begin{equation}\label{md}
\begin{split}
\|U(t)\|_{L^2\rightarrow L^2}&\leq C,\quad t\in \mathbb{R},\\
\|U(t)U(s)^*f\|_{L^\infty}&\leq
Ch^{-\alpha}(h+|t-s|)^{-\sigma}\|f\|_{L^1}.
\end{split}
\end{equation}
Then for every pair $q,p\in[1,\infty]$ such that $(q,p,\sigma)\neq
(2,\infty,1)$ and
\begin{equation*}
\frac{1}{q}+\frac{\sigma}{p}\leq\frac\sigma 2,\quad q\ge2,
\end{equation*}
there exists a constant $\tilde{C}$ only depending on $C$, $\sigma$,
$q$ and $r$ such that
\begin{equation*}
\Big(\int_{I}\|U(t) u_0\|_{L^r}^q dt\Big)^{\frac1q}\leq \tilde{C}
\Lambda(h)\|u_0\|_{L^2}
\end{equation*}
where $\Lambda(h)=h^{-(\alpha+\sigma)(\frac12-\frac1p)+\frac1q}$.
\end{proposition}

\begin{proof}
This is an analogue of the semiclassical Strichartz
estimates for Schr\"odinger in \cite{KTZ, Zworski}. We refer to \cite{Zhang} for the proof.
\end{proof}

Now we prove the Strichartz estimates \eqref{stri:w}. Recall $\varphi$ in \eqref{LP-dp} and
Littlewood-Paley frequency cutoff $\varphi_k(\sqrt{H_{\alpha, B_0}})$, for each $k\in\Z$, we
define
\begin{equation*}
u_k(t,\cdot)=\varphi_k(\sqrt{H_{\alpha, B_0}})u(t,\cdot).
\end{equation*}
where $u(t,x)$ is the solution of \eqref{eq:wave}.
Then,  for each $k\in\Z$,  $u_k(t,x)$ solves the Cauchy problem
\begin{equation*}
\partial_{t}^2u_k+H_{\alpha, B_0} u_k=0, \quad u_k(0)=f_k(z),
~\partial_tu_k(0)=g_k(z),
\end{equation*}
where $f_k=\varphi_k(\sqrt{H_{\alpha, B_0}})u_0$ and
$g_k=\varphi_k(\sqrt{H_{\alpha, B_0}})u_1$. Since $(q, p)\in \Lambda_s^W$ in definition \ref{ad-pair}, then $q, p\geq2$. Thus,  by using the square-function
estimates \eqref{square} and the Minkowski inequality, we obtain
\begin{equation}\label{LP}
\|u(t,x)\|_{L^q(I;L^p(\R^2))}\lesssim
\Big(\sum_{k\in\Z}\|u_k(t,x)\|^2_{L^q(I;L^p(\R^2))}\Big)^{\frac12},
\end{equation}
where $I=[0,T]$.
Denote the half-wave propagator $U(t)=e^{it\sqrt{H_{\alpha, B_0}}}$, then
we write
\begin{equation}\label{sleq}
\begin{split}
u_k(t,z)
=\frac{U(t)+U(-t)}2f_k+\frac{U(t)-U(-t)}{2i\sqrt{H_{\alpha, B_0}}}g_k.
\end{split}
\end{equation}
By using \eqref{LP} and \eqref{sleq}, we complete the proof of \eqref{stri:w} after taking summation in $k\in\Z$ if we could prove
\begin{proposition}\label{lStrichartz} Let
$f=\varphi_k(\sqrt{H_{\alpha, B_0}})f$ for $\varphi_k$ in \eqref{LP-dp} and $k\in\Z$. Then
\begin{equation}\label{lstri}
\|U(t)f\|_{L^q(I;L^p(\R^2))}\leq C_T
2^{ks}\|f\|_{L^2(\R^2)},
\end{equation}
where the admissible pair $(q,p)\in [2,+\infty]\times [2,+\infty)$ and $s$ satisfy
\eqref{adm} and \eqref{scaling}.
\end{proposition}

\begin{proof}
Since $f=\varphi_k(\sqrt{\mathrm{H}})f$, then $$U(t)f=\varphi_k(\sqrt{H_{\alpha, B_0}})e^{it\sqrt{H_{\alpha, B_0}}} f:=U_k f.$$
By using the spectral theorem, we see
\begin{equation*}
\|U_k(t)f\|_{L^2(\R^2)}\leq C\|f\|_{L^2(\R^2)}.
\end{equation*}
By using \eqref{est: mic-decay2}, we obtain
\begin{equation*}
\begin{split}
\|U_k(t)U_k^*(s)f\|_{L^\infty (\R^2)}&=\|U_k(t-s)f\|_{L^\infty (\R^2)}\\
&\leq C_T 2^{\frac32 k}\big(2^{-k}+|t-s|\big)^{-\frac12}\|f\|_{L^1(\R^2)},
\end{split}
\end{equation*}
Then the estimates \eqref{md}
for $U_{k}(t)$ hold for $\alpha=3/2$, $\sigma=1/2$ and
$h=2^{-k}$. Hence, Proposition \ref{prop:semi} gives
\begin{equation*}
\|U(t)f\|_{L^q(I;L^p(\R^2))}=\|U_k(t)f\|_{L^q(I;L^p(\R^2))}\leq C_T
2^{k[2(\frac12-\frac1p)-\frac1q]} \|f\|_{L^2(\R^2)}.
\end{equation*}
which implies \eqref{lstri} since $s=2(\frac12-\frac1p)-\frac1q$.
\end{proof}

\begin{center}

\end{center}

\end{document}